\documentclass[11pt,thmsa,a4paper,reqno]{article}

\usepackage{scrextend}
\changefontsizes{12pt}
\usepackage[all]{nowidow}
\usepackage[utf8x]{inputenc}
\usepackage{amscd,amssymb,amsmath}
\usepackage{amsrefs}
\usepackage{amsthm}
\usepackage{thmtools}
\usepackage[scr,scaled=1.1]{rsfso}

\usepackage{mathptmx}
\usepackage[T1]{fontenc}

\usepackage{bm}
\usepackage{bbm}

\usepackage{graphicx}
\usepackage{resizegather}
\usepackage{arydshln}
\usepackage{floatrow}
\usepackage{soul}
\usepackage{relsize,etoolbox}
 \usepackage[all,cmtip,2cell]{xy}
\UseAllTwocells
\xyoption{2cell}

\usepackage{tikz}
\usepackage{extarrows}
\usepackage{amssymb}

\numberwithin{equation}{section}

\usepackage[pdftex,
                paper=a4paper,
                portrait=true,
                textwidth=165mm,
                textheight=235mm,
                tmargin=2.5cm,
                marginratio=1:1]{geometry}

\newtheoremstyle{mydef}
  {}		
  {}		
  {}		
  {}		
  {\scshape}	
  {. }		
  { }		
  {\thmname{#1}\thmnumber{ #2}\thmnote{ #3}}	

\newtheorem{theorem}{Theorem}[section]
\newtheorem*{theorem*}{Theorem}
\newtheorem*{theoremA}{Theorem A}
\newtheorem*{theoremB}{Theorem B}
\newtheorem{proposition}[theorem]{Proposition}
\newtheorem*{proposition*}{Proposition}
\newtheorem{lemma}[theorem]{Lemma}
\newtheorem*{lemma*}{Lemma}

\newtheorem*{corollary*}{Corollary}
\theoremstyle{definition}

\theoremstyle{remark}

\newtheorem*{acknowledgements}{{\bf Acknowledgements}}
\newtheorem*{NotationConventions}{{\bf Notation and conventions}}

\DeclareMathAlphabet{\mathsfit}{T1}{\sfdefault}{\mddefault}{\sldefault}
\SetMathAlphabet{\mathsfit}{bold}{T1}{\sfdefault}{\bfdefault}{\sldefault}

\newcommand{\CC}{\ensuremath{\mathbbmss{C}}}
\newcommand{\NN}{\ensuremath{\mathbbmss{N}}}

\newcommand{\RR}{\ensuremath{\mathbbmss{R}}}
\newcommand{\TT}{\ensuremath{\mathbbmss{T}}} 
\newcommand{\ZZ}{\ensuremath{\mathbbmss{Z}}}

\newcommand{\Ccal}{\mathscr{C}}
\newcommand{\Dcal}{\mathscr{D}}
\newcommand{\Ecal}{\mathscr{E}}

\newcommand{\Hcal}{\mathscr{H}}
\newcommand{\Ical}{\mathscr{I}}

\newcommand{\Wcal}{\mathscr{W}}
\newcommand{\CW}{\mathscr{CW}}

\newcommand{\gfrak}{\mathfrak{g}}

\def\W{{\rm W}}

\newcommand{\ad}{\operatorname{ad}}

\newcommand{\tr}{\operatorname{tr}}
\newcommand{\cosk}{\operatorname{cosk}}
\newcommand{\Cosk}{\operatorname{Cosk}}
\newcommand{\Ex}{\operatorname{Ex}}

\newcommand{\End}{\operatorname{End}}
\newcommand{\Hom}{\operatorname{Hom}}

\newcommand{\CE}{\operatorname{CE}}
\newcommand{\dW}{d_{\W}}
\newcommand{\dCE}{\delta_{\CE}}
\newcommand{\dV}{\delta_{V}}
\newcommand{\dA}{\delta_{A}}
\newcommand{\Loc}{\operatorname{\mathbf{Loc}}}

\newcommand{\DGVect}{\operatorname{\mathbf{DGVect}}}
\newcommand{\DGMod}{\operatorname{\mathbf{DGMod}}}
\newcommand{\DGRep}{\operatorname{\mathbf{DGRep}}}
\newcommand{\InfLoc}{\operatorname{\mathbf{Inf}\,\mathbf{Loc}}}

\newcommand{\Rep}{\operatorname{\mathbf{Rep}}}
\newcommand{\uH}{\operatorname{H}}

\newcommand{\Ob}{\operatorname{Ob}}
\newcommand{\A}{A}
\renewcommand{\L}{L}
\newcommand{\s}{\mathsf{s}}
\renewcommand{\u}{\mathsf{u}}
\newcommand{\G}{\mathrm{G}}
\newcommand{\V}{\mathrm{V}}
\newcommand{\C}{\operatorname{C}}
\newcommand{\uT}{\operatorname{T}}

\def\bas{{\rm bas}}
\def\hor{{\rm hor}}
\def\inv{{\rm inv}}
\def\id{{\rm id}}
\def\pr{{\rm pr}}
\def\DR{{\rm DR}}

\def\uS{{\rm S}}
\def\Set{{\mathbf{Set}}}
\def\sSet{{\mathbf{sSet}}}

\DeclareFontFamily{U}{matha}{\hyphenchar\font45}
\DeclareFontShape{U}{matha}{m}{n}{
      <5> <6> <7> <8> <9> <10> gen * matha
      <10.95> matha10 <12> <14.4> <17.28> <20.74> <24.88> matha12
      }{}
\DeclareSymbolFont{matha}{U}{matha}{m}{n}
\DeclareFontSubstitution{U}{matha}{m}{n}
\DeclareMathSymbol{\abxcup}{\mathbin}{matha}{'131}

\newcommand*{\sbullet}{\raisebox{0.1ex}{\scalebox{0.6}{$\bullet$}}}

\DeclareMathAlphabet{\mathbfit}{OML}{cmm}{b}{it}

\begin{document}

\title{Chern-Weil theory for $\infty$-local systems}

\author{Camilo Arias Abad\footnote{Escuela de Matem\'{a}ticas, Universidad Nacional de Colombia Sede Medell\'{i}n, email: camiloariasabad@gmail.com}, Santiago Pineda Montoya\footnote{Escuela de Matem\'{a}ticas, Universidad Nacional de Colombia Sede Medell\'{i}n, email: sapinedamo@unal.edu.co}, and Alexander Quintero V\'{e}lez\footnote{Escuela de Matem\'{a}ticas, Universidad Nacional de Colombia Sede Medell\'{i}n, email: aquinte2@unal.edu.co}}

 \maketitle

\begin{abstract}
Let $G$ be a compact connected Lie group. We show that the category $\Loc_\infty(BG)$ of $\infty$-local systems on the classifying space of $G$, can be described infinitesimally as the category  $\InfLoc_{\infty}(\gfrak)$ of basic $\gfrak$-$\L_\infty$ spaces. Moreover, we show that, given a principal bundle $\pi \colon P \to X$ with structure group $G$ and any connection $\theta$ on $P$, there is a DG functor  
$$
\CW_{\theta} \colon \InfLoc_{\infty}(\gfrak) \longrightarrow \Loc_{\infty}(X),
$$
which corresponds to the pullback functor by the classifying map of $P$. The DG functors associated to different connections are related by an $\A_\infty$-natural isomorphism. This construction provides a categorification of the Chern-Weil homomorphism, which is recovered by applying the functor $\CW_{\theta}$ to the endomorphisms of the constant local system. 
\end{abstract}

\tableofcontents

\section{Introduction}
Given a compact connected Lie group $G$ with Lie algebra $\gfrak$, the cohomology of the classifying space $BG$ is isomorphic to the algebra of invariant polynomials $(\uS^{\sbullet} \gfrak^*)_{\inv}$ on $\gfrak$. Moreover, given a principal bundle $\pi\colon P \to X$ with structure group $G$ and any connection $\theta$ on $P$, the Chern-Weil homomorphism is an algebra map
$$
c_{\theta *}\colon (\uS^{\sbullet} \gfrak^*)_{\inv} \to \uH^{\sbullet}_{\mathrm{DR}}(X),
$$
which can be described geometrically in terms of the curvature, and corresponds to the map induced in cohomology by the classifying map $f\colon X \to BG$ of $P$. Our goal is to categorify this construction by replacing the cohomology of $BG$ by a more abstract invariant, the DG category $\Loc_\infty(BG)$ of $\infty$-local systems on $BG$.

 A local system on a topological space $X$ is a representation of the fundamental groupoid of $X$. By replacing the fundamental groupoid $\pi_1 X$ of $X$ by the $\infty$-groupoid, $\pi_{\infty} X_{\sbullet}$, one obtains the theory of $\infty$-local systems. An $\infty$-local system on $X$ is a representation (up to homotopy) of $\pi_{\infty} X_{\sbullet}$. Just like ordinary local systems can be described in several different ways, for instance, as flat vector bundles, the DG category of $\infty$-local systems also admits various descriptions.
 \begin{center}
  \begin{tabular}{| c | c | }
\hline 
   {\bf Point of view}   & {\bf $\infty$-local system}   \\ \hline
   Infinitesimal & Flat superconnection \\ 
 Simplicial & Representation of $\pi_{\infty}X_{\sbullet}$\\ 
 Topological & Representation of $\C_{\sbullet}(\Omega^{\mathrm{M}}X) $ \\   
\hline
  \end{tabular}
\end{center}
Here $\C_{\sbullet}(\Omega^{\mathrm{M}}X)$ denotes the algebra of singular chains on the based Moore loop space of $X$. Each of these notions of $\infty$-local system can be organized into a DG category, and these categories are quasi-equivalent. The proofs of these results can be found in \cite{Holstein2015,Block-Smith2014,Abad-Schatz2013,abadflat}.  The equivalence between the ``infinitesimal'' and the ``simplicial'' points of view is known as the higher Riemann-Hilbert correspondence.

In analogy with the Chern-Weil construction, we introduce a DG category $ \InfLoc_{\infty}(\gfrak)$, which is the infinitesimal counterpart of $\Loc_\infty(BG)$. This DG category is defined as follows. Given a Lie algebra $\gfrak$, we consider the DG Lie algebra $\TT \gfrak$, which is universal for the Cartan relations. The DG category $ \InfLoc_{\infty}(\gfrak)$ is a certain subcategory of the DG category of $\L_\infty$-representations of $\TT \gfrak$. 

Our first theorem is the construction of a DG functor that extends the Chern-Weil homomorphism.

\begin{theoremA}\label{theoremA}
Let $G$ be a Lie group and let $\pi \colon P \to X$ be a principal bundle with structure group $G$. Then, for any connection $\theta$ on $P$, there is a natural DG functor 
$$
\CW_{\theta} \colon \InfLoc_{\infty}(\gfrak) \longrightarrow \Loc_{\infty}(X).
$$
Moreover, for any two connections $\theta$ and $\theta'$ on $P$, there is an $A_{\infty}$-natural isomorphism between $\CW_{\theta}$ and $\CW_{\theta'}$. 
\end{theoremA}

Our second theorem states that, as expected, the functor $\CW_{\theta}$ corresponds to the pullback functor by the classifying map.

\begin{theoremB}
Given a compact connected Lie group $G$, there is a natural $\A_{\infty}$-functor 
$$
\Wcal \colon \InfLoc_{\infty}(\gfrak) \to \Loc_{\infty}(BG)
$$
which is an $\A_{\infty}$-quasi-equivalence. Moreover, if we let $\pi \colon P \to X$ be a principal bundle with structure group $G$, $\theta$ be any connection on $P$, and $f \colon X \to BG_n$ be the classifying map of $P$, then there exists an $\A_{\infty}$-natural isomorphism between the $\A_{\infty}$-functors $\Ical \circ \CW_{\theta}$ and $(\varphi_n \circ f)_{\sbullet}^* \circ \Wcal$ from $ \InfLoc_{\infty}(\gfrak)$ to $\Rep_{\infty}(\pi_{\infty}X_{\sbullet})$. Here $\Ical$ is the integration $\A_\infty$-functor provided by the higher Riemann-Hilbert correspondence, and $\varphi_n$ is the canonical map from $BG_n$ to $BG$.
\end{theoremB}

Let us explain the relationship between our results and the DG category of modules over the algebra $\C_{\sbullet}(G)$ of singular chains on $G$. The correspondence between representations of $\TT \gfrak$ and modules over $\C_{\sbullet}(G)$ was studied in \cite{arias2019singular, abad2020singular}. In these works it is proved that, for simply connected $G$, the category of representations of $\TT \gfrak$ is equivalent to the category of ``sufficiently smooth'' modules over $\C_{\sbullet}(G)$. Moreover, if $G$ is also compact, the DG enhancements of these categories are also $\A_\infty$-quasi-equivalent. Precisely, the following is proved in \cite{abad2020singular}.

\begin{theorem*}
Suppose that $G$ is compact and simply connected. The DG categories $\DGRep(\TT\gfrak)$ of representations of $\TT \gfrak$ and $\DGMod(\C_{\sbullet}(G))$  of modules over $\C_{\sbullet}(G)$ are $\A_{\infty}$-quasi-equivalent. 
\end{theorem*}

Given that the DG category  $\DGRep(\TT\gfrak)$ is a subcategory of $\InfLoc_{\infty}(\gfrak)$, one concludes that the category of modules over $\C_{\sbullet}(G)$ consists of $\infty$-local systems on $BG$. This should be expected on topological grounds because $G$ is $\A_{\infty}$ equivalent to $\Omega^{\mathrm{M}} BG$, and therefore,  $\C_{\sbullet}(G)$ and $\C_{\sbullet}(\Omega^{\mathrm{M}} BG)$ are homotopy equivalent. However, the argument we provide is infinitesimal, in terms of the Lie algebra $\gfrak$.

\begin{NotationConventions}
All vector spaces and algebras are defined over the field of real numbers $\RR$.
All cochain complexes are cohomologically graded, that is, the differential increases the degree by $1$.
 If $V = \bigoplus_{k \in \ZZ} V^k$ is a graded vector space, we denote by $\s V$ its suspension, which is the graded vector space with grading defined by
$$
(\s V)^{k} = V^{k-1},
$$
and by $\u V$ its unsuspension, which is the graded vector space with grading defined by
$$
(\u V)^{k} = V^{k+1}.
$$

The symmetric group on $n$ letters is denoted by $\mathfrak{S}_{n}$. If $i_1 +\dots+i_d=n$, then $\mathfrak{S}_{i_1, \dots, i_d}$ is the set of $(i_1, \dots,i_d)$-shuffle permutations in  $\mathfrak{S}_{n}$. That is, the set of permutations that preserve the order on each block $\{i_1+\dots+i_l+1, \dots, i_1+ \dots +i_{l+1}\}$.

The symmetric algebra $\uS^{\sbullet} V$ of a graded vector space $V$ is the quotient of the tensor algebra $\uT^{\sbullet} V$ by the graded ideal generated by elements of the form $u \otimes v - (-1)^{\vert u \vert \vert v \vert} v \otimes u$, for homogeneous elements $u,v \in V$. We write $v_1 \odot \dots \odot v_k$ for the element represented by $v_1 \otimes \cdots \otimes v_k$ in the quotient space $\uS^{k} V$. Given a permutation $\sigma \in \mathfrak{S}_n$, we denote by $\varepsilon(\sigma;v_1,\dots,v_1)$ the graded Koszul sign, which is defined via 
$$
v_{\sigma(1)}\odot \cdots \odot v_{\sigma(n)} = \varepsilon(\sigma;v_1,\dots,v_n) v_1\odot \cdots \odot v_n.
$$ 
We denote by $\bigodot^{\sbullet}V$ the cocommutative coalgebra whose underlying vector space is $\uS^{\sbullet}V$ and has coproduct given by
$$
\Delta(v_1\odot \cdots \odot v_n)=\sum_{i+j= n} \sum_{\sigma \in \mathfrak{S}_{i,j}} \varepsilon(\sigma;v_1,\dots,v_n) v_{\sigma(1)}\odot \cdots \odot v_{\sigma(i)} \otimes v_{\sigma(i+1)}\odot \cdots \odot v_{\sigma(n)}.
$$
The coalgebra $\bigodot^{\sbullet}V$ is cocommutative, counital and coaugmented, with coaugmentation $c\colon \RR \to \bigodot^{\sbullet}V$ given by $ 1 \mapsto 1 \in \bigodot^{0}V=\RR$, and counit $\epsilon\colon \bigodot^{\sbullet}V \to \RR$ given by the natural projection. We will refer to the coalgebra $\bigodot^{\sbullet}V$ as the cocommutative coalgebra cogenerated by $V$. Given another graded vector space $W$ and a linear map of degree zero $\varphi\colon \bigodot^{\sbullet} W \to V$ there is unique coaugmented coalgebra map $\bar{\varphi}\colon \bigodot^{\sbullet}W\to \bigodot^{\sbullet} V$ such that the following diagram commutes:
\[ \xymatrix{
\bigodot^{\sbullet}W \ar[r]^{\bar{\varphi}} \ar[rd]_{\varphi}& \bigodot^{\sbullet}V\ar[d]^{\pi}\\
& V.
}
\]
Here $\pi\colon \bigodot^{\sbullet}V \to V$ is the natural projection map. The map $\overline{\varphi}$ is given explicitly by the formula
\begin{align*}
 &\bar{\varphi}(v_1 \odot \dots \odot v_n) \\
 &\qquad=
\sum_{i_1 + \cdots + i_{p} = n} \sum_{\sigma \in \mathfrak{S}_{i_1,\dots,i_p}} \varepsilon(\sigma;x_1,\dots,x_n) \frac{1}{p!} \prod_{k=1}^{p}\varphi (x_{\sigma(i_1 + \cdots + i_{k-1} + 1)}\odot \cdots \odot x_{\sigma(i_1 + \cdots + i_{k})}).
\end{align*}

We write $\Delta_n$ for the standard geometric $n$-simplex 
$$
\Delta_n = \{ \mathbf{t}=(t_1,\dots,t_n) \in \RR^{n} \mid 1 \geq t_1 \geq \cdots \geq t_n \geq 0 \}.
$$
Given a simplicial set $K_{\sbullet}$, its geometric realization is denoted by $\vert K_{\sbullet} \vert$. We write $[x,\mathbf{t}]$ to designate the equivalence class of a point $(x,\mathbf{t}) \in K_p \times \Delta_p$ with respect to the equivalence relation that defines $\vert K_{\sbullet} \vert$. The geometric realization of a simplicial map $f_{\sbullet} \colon K_{\sbullet} \to L_{\sbullet}$ is the map $\vert f_{\sbullet} \vert \colon \vert K_{\sbullet} \vert \to \vert L_{\sbullet} \vert$ given by $[x,\mathbf{t}] \mapsto [f_p(x),\mathbf{t}]$. 

All principal bundles are defined in terms of a free right action $\sigma \colon P \times G \to P$ of a Lie group $G$ on a smooth manifold $P$. For each $g \in G$, we write $\sigma_g \colon P \to P$ for the diffeomorphism given by $p  \mapsto \sigma(p,g) = p \cdot g$.

If $E = \bigoplus_{k \in \ZZ} E^{k}$ is a graded vector bundle over a  smooth manifold, we define
$$
\Omega^{\sbullet}(M,E) = \Gamma(\Lambda^{\sbullet} T^* M \otimes E).
$$
This space is graded by the total degree
$$
\Omega^{\sbullet}(M,E)^{n} = \bigoplus_{p + q = n} \Omega^{p}(M, E^{q}).
$$
Given an element $\omega \in \Omega^{p}(M, E^{q})$ we will say that $\omega$ is of partial degree $p$.
\end{NotationConventions}

\begin{acknowledgements}
We would like to acknowledge the support of Colciencias through  their grant {\it Estructuras lineales en topolog\'ia y geometr\'ia}, with contract number FP44842-013-2018.  We also thank  the Alexander von Humboldt foundation which supported our work through the Humboldt Institutspartnerschaftet { \it Representations of Gerbes and higher holonomies}. 
We are grateful to Manuel Rivera for clarifying conversations about DG categories. 
\end{acknowledgements}


\section{Preliminaries}
In this section we recall some terminology and results concerning DG Lie algebras, DG categories and $\infty$-local systems.The books by Guillemin-Sternberg and Meinrenken \cite{Meinrenken2013,guillemin} discuss the construction of the Weil algebra of a Lie algebra. For a detailed exposition of $\L_\infty$-algebras and morphisms we recommend \cite{reinhold2019algebras}. The paper by Keller \cite{Keller2006} provides an excellent introduction to DG categories. 
Our conventions on $\infty$-local systems and the higher Riemann-Hilbert correspondence are taken from \cite{Block-Smith2014,AriasAbad-Crainic2012,Abad-Schatz2013}.

\subsection{DG Lie algebras and $\L_{\infty}$-morphisms}\label{sec:2.1}
 A \emph{DG Lie algebra} (where DG stands for ``differential graded'') is a graded Lie algebra $L = \bigoplus_{i \in \ZZ} L^{i}$ equipped with a linear map $d \colon L \to L$ of degree $1$ with $d^2 = 0$, such that the Leibniz  rule
$$
d [x,y]= [dx, y] + (-1)^{\vert x \vert} [x, dy],
$$ 
holds for  homogeneous elements $x,y \in L$. A \emph{strict morphism of DG Lie algebras} $\phi \colon L \to L'$ is a homomorphism of graded Lie algebras with $\phi \circ d = d' \circ \phi$.

A basic example of a DG Lie algebra is provided by the linear endomorphisms $\End(V)$ of a cochain complex of vector spaces $V$. The bracket operation is given by the graded commutator and the differential is defined on homogeneous elements $f \in \End(V)$ by  
$$
d(f) = \delta_V \circ f - (-1)^{\vert f \vert} f \circ \delta_V,
$$
where here $\delta_V$ is the differential of $V$. 

Given a commutative DG algebra $A$ and a DG Lie algebra $L$, the tensor product $A \otimes L$ has the structure of a DG Lie algebra. The bracket operation and the differential on $A \otimes L$ are
\begin{align*}
[a \otimes x, b \otimes y] &= (-1)^{\vert x \vert \vert b \vert} ab \otimes [x,y], \\
d (a \otimes x) &=da \otimes x +(-1)^{\vert a \vert}a \otimes dx,
\end{align*}
for homogeneous $a,b \in A$ and $x, y \in L$. 

Let $L$ be a DG Lie algebra and consider the cocommutative coalgebra $\bigodot^{\sbullet}(\u L)$. The differential and the bracket of $L$ can be encoded in a single coderivation $D$ on $\bigodot^{\sbullet}(\u L)$. Explicitly, this coderivation is defined by setting for $x_1,\dots, x_n \in L$, 
\begin{gather*}
D (\u x_1 \odot \cdots \odot \u x_n) = \sum_{i=1}^{n} \varepsilon(\sigma_{i},\u x_1,\dots,\u x_n) \u(dx_i)\odot \u x_1\odot \cdots \odot \u x_{i-1} \odot \u x_{i+1}\odot \cdots  \odot \u x_n  \\
 + \sum_{ 1\leq i < j \leq n } \varepsilon(\sigma_{ij},\u x_1,\dots,\u x_n)(-1)^{\vert v_i \vert}  \u[x_i,x_j]\odot \u x_1 \cdots \odot \u x_{i-1} \odot \u x_{i+1} \odot \cdots  \odot \u x_{j-1}\odot  \u x_{j+1}\odot \cdots  \odot \u x_n,
\end{gather*}
where $\sigma_i \in \mathfrak{S}_n$ is the permutation which sends $1$ to $i$, subtracts one from the integers $2,\dots, i$ and fixes the other integers,  and $\sigma_{ij} \in \mathfrak{S}_n$ is the permutation which sends $1$ to $i$, sends $2$ to $j$, removes two from the integers $2,\dots, i+1$, removes one from the integers $i+2,\dots, j$ and fixes the other integers. This derivation satisfies $D^2 = 0$, so that the coalgebra $\bigodot^{\sbullet}(\u L)$ has the structure of a differential graded coalgebra. Moreover, for any two DG Lie algebras $L$ and $L'$, a linear map $\phi \colon L \to L'$ of degree $0$ is a morphism of DG Lie algebras if and only if $\overline{\u \phi} \circ D = D' \circ \overline{\u \phi}$, where $D$ and $D'$ are the codifferentials on $\bigodot^{\sbullet}(\u L)$ and $\bigodot^{\sbullet}(\u L')$, respectively, and $\overline{\u \phi}:\bigodot^{\sbullet} (\u L) \to \bigodot^{\sbullet}(\u L')$ is the coalgebra map associated to the linear map $\u \phi : \u L \to \u L'$.

We will consider a more general notion of morphism between DG Lie algebras, that of $\L_{\infty}$-morphism. Let $L$ and $L'$ be  DG Lie algebras with corresponding codifferentials  $D$ and $D'$. A  linear map $\Phi \colon \bigodot^{\sbullet}(\u L) \to \u L'$ of degree $0$ is called an \emph{$\L_{\infty}$-morphism} between $L$ and $L'$ if 
$$
\bar{\Phi} \circ D = D' \circ \bar{\Phi}.
$$
Such an $\L_\infty$-morphism can be written as the sum
$$
\Phi=\Phi_1 + \Phi_2+\Phi_3+ \dots,
$$
where $\Phi_k$ is the restriction of $\Phi$ to the vector space $\textstyle{\bigodot}^k(\u L)$.
Strict morphisms of DG Lie algebras are a particular instance of $\L_{\infty}$-morphisms that correspond the case where
$\Phi_k=0$ for $k>1$ . 

Let $L$ be a DG Lie algebra. An element $x \in L$ of degree $1$ is said to be a \emph{Maurer-Cartan element} if
$$
dx + \frac{1}{2}[x,x] = 0.
$$
This equation, which describes an abstract form of ``flatness'', is known as the \emph{Maurer-Cartan equation}. Given a differential graded Lie algebra $L$, the dual vector space of the differential graded coalgebra $\bigodot^{\sbullet}(\u L)$ is a differential graded Lie algebra known as the \emph{Chevalley-Eilenberg DG algebra} of L. We shall denote it by $\CE(L)$ 
and write $\dCE = D^*$ for the corresponding differential. In the special case where $L$ is a Lie algebra, the definition reduces to that of the usual Chevalley-Eilenberg complex which computes the Lie algebra cohomology of $L$.

The following result, which a consequence of Proposition~3.3 of \cite{Metha-Zambon2012}, will be used throughout the text.

\begin{proposition}\label{prop:2.1}
Let $L$ and $L'$ be  DG Lie algebras, and suppose that $L$ is finite dimensional. Then, there is a natural identification between the set of $L_{\infty}$-morphisms from $L$ to $L'$ and the set of Maurer-Cartan elements of $\CE(L) \otimes L'$.  
\end{proposition}

The identification goes through the following sequence of isomorphisms of vector spaces:
\begin{align*}
\left[\CE(L) \otimes L'\right]^1 &\cong \left[\left(\bigoplus_{k \geq 0}\textstyle{\bigodot^{k}} (\u L)\right)^*\otimes L' \right]^1 \\
&\cong \prod_k \left[\left(\textstyle{\bigodot^{k}}(\u L) \right)^*\otimes L' \right]^1\\
&\cong\Hom^1\left( \textstyle{\bigodot^{\sbullet} }(\u L), L'\right) \\
&\cong \Hom^0\left( \textstyle{\bigodot^{\sbullet} }(\u L), \u L'\right).
\end{align*}
Maurer-Cartan elements live on the space on the first vector space, $\L_\infty$-morphisms live on the last vector space, and
the corresponding conditions map to one another.

\subsection{DG categories, DG functors and $\A_{\infty}$-functors}\label{sec:2.2}
A \emph{DG category} is a linear category $\Ccal$ such that for every two objects $A$ and $B$ the space of arrows $\Hom_{\Ccal}(A,B)$ is equipped with a structure of a cochain complex of vector spaces, and for every three objects $A$, $B$ and $C$ the composition map $\Hom_{\Ccal}(B,C) \otimes \Hom_{\Ccal}(A,B) \to \Hom_{\Ccal}(A,C)$ 
is a morphism of cochain complexes. Thus, by definition, 
\[
\Hom_{\Ccal}(A,B) = \bigoplus_{n \in \ZZ} \Hom_{\Ccal}^{n}(A,B)
\]
is a graded vector space with a differential $d \colon \Hom_{\Ccal}^{n}(A,B) \to \Hom_{\Ccal}^{n+1}(A,B)$. The elements $f \in \Hom_{\Ccal}^{n}(A,B)$ are called \emph{homogeneous of degree $n$}, and we write $\vert f \vert = n$. We denote the set of objects of $\Ccal$ by $\Ob \Ccal$.

The fundamental example of a DG category is the category of cochain complexes of vector spaces, which we denote by $\DGVect$. Its objects are cochain complexes of vector spaces and the morphism spaces $\Hom_{\DGVect}(V,W)$ are endowed with the differential defined as
$$
d (f) = \delta_{W} \circ f - (-1)^n f \circ \delta_{V},
$$
for any homogeneous element $f$ of degree $n$.   

Let $\Ccal$ be a DG category and let $A \in \Ob \Ccal$. Given a closed morphism $f \in \Hom_{\Ccal}^{0}(B,C)$ we define $f_* \colon \Hom_{\Ccal}(A,B) \to \Hom_{\Ccal}(A,C)$ by $f_{*}(g) = f \circ g$ for $g \in \Hom_{\Ccal}(A,B)$. It is not difficult to see that $f_*$ is a morphism of cochain complexes. Similarly, if we define $f^* \colon \Hom_{\Ccal}(C,A) \to \Hom_{\Ccal}(B,A)$ by $f^*(h) = h \circ f$ for $h \in \Hom_{\Ccal}(C,A)$, then $f^*$ es a morphism of cochain complexes. 

Given a DG category $\Ccal$ one defines an ordinary category $\mathbf{Ho}(\Ccal)$ by keeping the same set of objects and replacing each $\Hom$ complex by its $0$th cohomology. We call $\mathbf{Ho}(\Ccal)$ the \emph{homotopy category} of $\Ccal$. 

If $\Ccal$ and $\Dcal$ are DG categories, a \emph{DG functor} $F \colon \Ccal \to \Dcal$ is a linear functor whose associated map for $A, B \in \Ob \Ccal$,
$$
F_{A,B} \colon \Hom_{\Ccal}(A,B) \to \Hom_{\Dcal}(F(A),F(B)),
$$ 
is a morphism of cochain complexes. Notice that any DG functor $F \colon \Ccal \to \Dcal$ induces an ordinary functor 
$$
\mathbf{Ho}(F) \colon \mathbf{Ho}(\Ccal) \to \mathbf{Ho}(\Dcal)
$$
between the corresponding homotopy categories. A DG functor $F \colon \Ccal \to \Dcal$ is said to be \emph{quasi fully faithful} if for every pair of objects $A, B \in \Ob \Ccal$ the morphism $F_{A,B}$ is a quasi-isomorphism. Moreover, the DG functor $F $ is said to be \emph{quasi essentially surjective} if $\mathbf{Ho}(F)$ is essentially surjective. A DG functor which is both quasi fully faithful and quasi essentially surjective is called a \emph{quasi-equivalence}. 

A morphism $f \in \Hom^0_{\Ccal}(B,C)$ is said to be a \emph{quasi-isomorphism} if it is closed and its equivalence class in $\mathbf{Ho}(\Ccal)$ is an isomorphism. The following lemma is an immediate consequence of the definition.

\begin{lemma}
Let  $f \in \Hom^0_{\Ccal}(B,C)$  be a quasi-isomorphism. Then, for any object $A \in \Ob \Ccal$, both $f_* \colon \Hom_{\Ccal}(A,B) \to \Hom_{\Ccal}(A,C)$ and $f^* \colon \Hom_{\Ccal}(C,A) \to \Hom_{\Ccal}(B,A)$ are quasi-isomorphisms.
\end{lemma}

\begin{proof}
Let us write $[f]$ for the equivalence class of $f$ in $\mathbf{Ho}(\Ccal)$. Then, by assumption, $[f]$ has an inverse $[g]\in \Hom_{\mathbf{Ho}(\Ccal)}(C,B).$ Let $g\in \Hom^{0}_{\Ccal}(C,B)$ be a representative of $[g]$. Then, there exists $x \in  \Hom^{-1}_{ \Ccal}(B,B)$ and $y \in \Hom^{-1}_{ \Ccal}(C,C)$ such that
\begin{align*}
 g \circ f &=\id_A + dx,\\
 f \circ g &= \id_B + dy.
\end{align*}
This implies that
\begin{align*}
g_* \circ f_* &=(g \circ f)_*=(\id_A +dx)_*, \\
f_* \circ g_* &=(f \circ g)_*=(\id_B +dy)_*.
\end{align*}
Since $(\id_A+dx)_*$ and $(\id_B+dy)_*$ induce the identity on cohomology, one concludes that $f_*$ is a quasi-isomorphism. Similarly,
\begin{align*}
f^* \circ g^* &=(g \circ f)^*=(\id_A +dx)^*, \\
g^* \circ f^* &=(f \circ g)^*=(\id_B +dy)^*,
\end{align*}
and since $(\id_A+dx)^*$ and $(\id_B+dy)^*$ induce the identity on cohomology, one gets that $f^*$ is a quasi-isomorphism.
\end{proof}

There is a more general notion of functor between DG categories, that of $\A_{\infty}$-functor, where the composition is preserved only up to an infinite sequence of coherence conditions. It will be useful to introduce first the Hochschild chain complex of a DG category. 

Let $\Ccal$ be a small DG category. The \emph{Hochschild cochain complex} of $\Ccal$  denoted $\mathrm{HC}(\Ccal)$ is the cochain complex defined as follows. As a vector space
$$
\mathrm{HC}(\Ccal)=\bigoplus_{A_0,\dots, A_n}\u \Hom_{\Ccal}(A_{n-1},A_n) \otimes \cdots \otimes \u \Hom_{\Ccal}(A_{0},A_1),
$$
where $A_0,\dots,A_n$ range through the objects of $\Ccal$. The differential $b$ is the sum of two components $b_1$ and $b_2$, which are given by the formulas
$$
b_1 (f_{n-1} \otimes \cdots \otimes f_0) = \sum_{i=0}^{n-1} (-1)^{\sum_{j=i+1}^{n-1}\vert f_{j} \vert  + n - i - 1}  f_{n-1} \otimes \cdots \otimes d f_i \otimes \cdots\otimes f_0
$$
and 
$$
b_2 (f_{n-1} \otimes \cdots \otimes f_0) = \sum_{i=0}^{n-2} (-1)^{\sum_{j=i+2}^{n-1}\vert f_{j} \vert + n - i } f_{n-1} \otimes \cdots \otimes (f_{i+1} \circ f_i) \otimes \cdots\otimes f_0
$$
for homogeneous elements $f_0 \in  \u\Hom_{\Ccal}(A_{0},A_1), \dots, f_{n-1} \in \u\Hom_{\Ccal}(A_{n-1},A_n)$.  Here $d$ denotes indistinctly the differential in any of the spaces $\Hom_{\Ccal} (A_{i},A_{i+1})$. It is easy to check that indeed $b^2=0$. 

Let $\Ccal$ and $\Dcal$ be DG categories. An \emph{$\A_{\infty}$-functor} $F \colon \Ccal \to \Dcal$ is the datum of a map of sets $F_0 \colon \Ob \Ccal \to \Ob \Dcal$ and a collection of $K$-linear maps of degree $0$
$$
F_n \colon \u \Hom_{\Ccal}(A_{n-1},A_n) \otimes \cdots \otimes \u \Hom_{\Ccal}(A_{0},A_1) \to \Hom_{\Dcal}(F_0(A_0),F_0(A_n))
$$
for every collection $A_0,\dots,A_n \in \Ob \Ccal$, such that the relation
\begin{align*}
b_1 \circ  F_n  + \sum_{i+j = n} b_2  \circ (F_i \otimes F_j ) = \sum_{i + j + 1= n} F_n \circ (\id^{\otimes i} \otimes b_1 \otimes \id^{\otimes j}) + \sum_{i + j + 2 = n} F_{n-1} \circ (\id^{\otimes i} \otimes b_2 \otimes \id^{\otimes j}) 
\end{align*} 
is satisfied for any $n \geq 1$. We also require that $F_1(\id_{A}) = \id_{F_0(A)}$ for all objects $A$ in $\Ccal$, as well as $F_n(f_{n-2} \otimes \cdots \otimes f_{i} \otimes \id_{A_{i}} \otimes f_{i-1} \otimes \cdots \otimes f_0 ) =0$ for any $n \geq 1$, any $0 \leq i \leq n-2$, and any chain of morphisms $f_0 \in  \u \Hom_{\Ccal}(A_{0},A_1), \dots, f_{n-2} \in \u  \Hom_{\Ccal}(A_{n-2},A_{n-1})$. 

The above relation when $n=1$ implies that $F_1$ is a morphism of cochain complexes. On the other hand, for $n=2$ we find that $F_1$ preserves the compositions on $\Ccal$ and $\Dcal$, up to a homotopy defined by $F_2$. In particular, a DG functor between $\Ccal$ and $\Dcal$ is identified with and $\A_{\infty}$-functor having $F_n = 0$ for $n \geq 2$. It also follows that $F_1$ induces an ordinary functor 
$$
\mathbf{Ho}(F_1) \colon  \mathbf{Ho}(\Ccal)  \to \mathbf{Ho}(\Dcal). 
$$
An $\A_{\infty}$-functor $F \colon \Ccal \to \Dcal$ is called $\A_{\infty}$-\emph{quasi fully faithful} if $F_1$ is a quasi-isomorphism of each pair of objects, $F$ is called $\A_{\infty}$-\emph{quasi essentially surjective} if $\mathbf{Ho}(F_1)$ is essentially surjective. Moreover, the $\A_{\infty}$-functor $F$ is called a $\A_{\infty}$-\emph{quasi-equivalence} if it is both quasi fully faithful and quasi essentially surjective. 

We need one more notion. Let $\Ccal$ and $\Dcal$ be DG categories and let $F \colon \Ccal \to \Dcal$ and $G \colon \Ccal \to \Dcal$ be DG functors. An \emph{$\A_{\infty}$-natural transformation}  $\lambda \colon F \Rightarrow G$ is the datum of a closed morphism $\lambda_0(X) \in \Hom_{\Dcal}^0 (F(A),G(A))$ for each $A \in \Ob \Ccal$ and a collection of $K$-linear maps of degree $0$
$$
\lambda_n \colon \u \Hom_{\Ccal}(A_{n-1},A_n) \otimes \cdots \otimes \u \Hom_{\Ccal}(A_{0},A_1) \to \Hom_{\Dcal}(F(A_0),G(A_n))
$$
for every collection $A_0,\dots,A_n \in \Ob \Ccal$, such that for all composable chains of homogeneous morphisms $f_0 \in  \u \Hom_{\Ccal}(A_{0},A_1), \dots, f_{n-1} \in \u \Hom_{\Ccal}(A_{n-1},A_n)$ the relation
\begin{align*}
G(f_{n-1}) \circ \lambda_{n-1}(f_{n-2}\otimes \cdots \otimes f_0) &- (-1)^{\sum_{i=1}^{n-1}\vert f_i \vert -n+1} \lambda_{n-1}(f_{n-1} \otimes\cdots \otimes f_1) \circ F(f_0) \\
&\qquad\qquad  = \lambda \left( b(f_{n-1} \otimes \cdots \otimes f_0) \right) - d \left(\lambda_n (f_{n-1} \otimes \cdots \otimes f_0) \right)
\end{align*}
is satisfied for any $n \geq 1$. The $\lambda$ on the right denotes the direct sum of the various $\lambda_n$. For $n=1$ this yields the condition
$$
G(f_0) \circ \lambda_0(A_0) - \lambda_0(A_1) \circ F(f_0) = \lambda_1 \left(d(f_0) \right) - d \left(\lambda_1(f_0)\right).
$$
Since the map $\lambda_1 \colon \u\Hom_{\Ccal}(A_{0},A_1) \to \Hom_{\Dcal}(F(A_{0}),G(A_1))$ has degree $-1$ when considered as a map defined over $\Hom_{\Ccal}(A_{0},A_1)$, this implies that the diagram
$$
\xymatrix{ F(A_0) \ar[r]^-{\lambda_0(A_0)} \ar[d]_-{F(f_0)} & G(A_0) \ar[d]^-{G(f_0)} \\ F(A_1) \ar[r]_-{\lambda_0(A_1)} & G(A_1)}
$$
commutes up to a homotopy given by $\lambda_1$. 

As usual, $\A_{\infty}$-natural transformations can be composed: if $F \colon \Ccal \to \Dcal$, $G \colon \Ccal \to \Dcal$ and $H \colon \Ccal \to \Dcal$ are three DG functors from the DG category $\Ccal$ to the DG category $\Dcal$, and $\lambda \colon F \Rightarrow G$ and $\mu \colon G \Rightarrow H$ are two $\A_{\infty}$-natural transformations, then the formula
$$
(\mu \circ \lambda)_n = \sum_{i=0}^n \mu_i \circ \lambda_{n-i}
$$
defines a new $\A_{\infty}$-natural transformation $\mu \circ \lambda \colon F \Rightarrow H$.  An \emph{$\A_{\infty}$-natural isomorphism} between functors from $\Ccal$ to $\Dcal$ is an $\A_\infty$-natural transformation $\lambda$
such that $\lambda_0(A)$ is an isomorphism for all $A \in \Ob \Ccal$.

We close this subsection with the following observation.

\begin{lemma}\label{lem:2.2}
Let $\Ccal$, $\Dcal$ and $\Ecal$ be DG categories and let $F\colon \Ccal \to \Dcal$, $G \colon \Ccal \to \Dcal$ and $H \colon \Dcal \to \Ecal$ be DG functors. Then for each $\A_{\infty}$-natural transformation $\lambda \colon F \Rightarrow G$ there is an induced $\A_{\infty}$-natural transformation $H \circ \lambda \colon H \circ F \Rightarrow H \circ G$. Moreover, if $\lambda$ is an $\A_{\infty}$-natural isomorphism, then $H \circ \lambda$ is an $\A_{\infty}$-natural isomorphism.
\end{lemma}

\begin{proof}
The formula for $H \circ \lambda$ reads
$$
(H \circ \lambda)_n = \sum_{k = 2}^{n+2} \sum_{i+j = k} H_{k-1} \circ (G^{\otimes(i-1)} \otimes \lambda_{n+2-k} \otimes F^{\otimes (j-1)}),
$$
for any $n \geq 1$. Let us check that this indeed defines an $\A_{\infty}$-natural transformation between $H \circ F$ and  $H \circ G$. For this purpose, let $m$ indistinctly denote the composition operation in $\Ccal$, $\Dcal$ or $\Ecal$. In this notation, what we need to show is that
$$
m \circ ((H \circ G) \otimes (H \circ \lambda)_{n-1}) - (-1)^n m \circ ( (H\circ \lambda)_{n-1} \otimes (H\circ F)) = (H\circ \lambda) \circ b - d \circ (H \circ \lambda)_n,
$$
wherein
 $$
b = \sum_{i=1}^{n-1} (-1)^{n-i-1} \id^{\otimes (n-i-1)} \otimes d \otimes \id^{\otimes i}  + \sum_{i=0}^{n-2} (-1)^{n-i} \id^{\otimes (n-i-2)} \otimes m \otimes \id^{\otimes i}.
$$
Let us start with the right-hand side. Using the definition of $H \circ \lambda$ and the above expression for $b$, this becomes 
\begin{align*}
&\sum_{l=0}^{n-1} (-1)^{n-l-1} \sum_{k=2}^{n+2} \sum_{i+j=k} H_{k-1}  \circ (G^{\otimes(i-1)} \otimes \lambda_{n+2-k} \otimes F^{\otimes (j-1)}) \circ (\id^{\otimes (n-l-1)} \otimes d \otimes \id^{\otimes l}) \\
&\quad \quad + \sum_{l=0}^{n-2} (-1)^{n-l} \sum_{k=2}^{n+1} \sum_{i+j=k} H_{k-1}  \circ (G^{\otimes(i-1)} \otimes \lambda_{n+1-k} \otimes F^{\otimes (j-1)}) \circ (\id^{\otimes (n-l-2)} \otimes m \otimes \id^{\otimes l}) \\
&\quad \quad - d \circ \left[ \sum_{k=2}^{n+2} \sum_{i+j=k} H_{k-1} \circ (G^{\otimes (i-1)} \otimes \lambda_{n+2-k} \otimes F^{\otimes (j-1)})\right] \\
&\quad = \sum_{k=2}^{n+2} \sum_{i+j=k} \left\{\sum_{l=0}^{n-1} (-1)^{n-l-1}  H_{k-1} \circ (G^{\otimes (i-1)} \otimes \lambda_{n+2-k} \otimes F^{\otimes (j-1)}) \circ (\id^{\otimes (n-l-1)} \otimes d \otimes \id^{\otimes l}) \right. \\
&\quad \quad\quad \phantom{\sum_{k=2}^{n+2}\sum_{i+j=k}} + \sum_{l=0}^{n-2}(-1)^{n-l}  H_{k-1} \circ (G^{\otimes (i-1)} \otimes \lambda_{n+1-k} \otimes F^{\otimes (j-1)}) \circ (\id^{\otimes (n-l-2)} \otimes m \otimes \id^{\otimes l}) \\
&\quad \quad\quad \left.\phantom{\sum_{k=2}^{n+2}\sum_{i+j=k}} - d \circ  H_{k-1} \circ (G^{\otimes (i-1)} \otimes \lambda_{n+2-k} \otimes F^{\otimes (j-1)}) \right\}.
\end{align*}
If in the right-hand side of the last equality we separate out the term $k = 2$ and, in the first sum over $l$, the terms in which $l$ varies over the set $\{i-1,\dots,n-j\}$, we get
\begin{align*}
 &\sum_{l=0}^{n-1}(-1)^{n-l-1} H_1 \circ \lambda_n \circ (\id^{\otimes (n-l-1)} \otimes d \otimes \id^{\otimes l}) \\
 &+ \sum_{l=0}^{n-2} H_1 \circ  \lambda_{n-1} \circ (\id^{\otimes (n-l-1)} \otimes m \otimes \id^{\otimes l}) + d \circ H_1  \circ  \lambda_n \\
 &+ \sum_{k=3}^{n+2} \sum_{i+j=k} \Bigg\{\sum_{\substack{ l=0 \\ l \not\in\{i-1,\dots,n-j\}}}^{n-2} (-1)^{n-l-1}  H_{k-1} \circ (G^{\otimes (i-1)} \otimes \lambda_{n+2-k} \otimes F^{\otimes (j-1)}) \circ (\id^{\otimes (n-l-1)} \otimes d \otimes \id^{\otimes l})  \\
 & \quad\quad \phantom{\sum_{k=2}^{n+2}\sum_{i+j=k}} + \sum_{ l=i-1}^{n-j} (-1)^{n-l-1}  H_{k-1} \circ \left[G^{\otimes (i-1)} \otimes (\lambda_{n+2-k} \circ (\id^{\otimes (n-l-i)} \otimes d \otimes \id^{\otimes l - j +1})) \otimes  F^{\otimes (j-1)}\right]\\
&\quad\quad \phantom{\sum_{k=2}^{n+2}\sum_{i+j=k}} + \sum_{l=0}^{n-2}(-1)^{n-l}  H_{k-1} \circ (G^{\otimes (i-1)} \otimes \lambda_{n+1-k} \otimes F^{\otimes (j-1)}) \circ (\id^{\otimes (n-l-2)} \otimes m \otimes \id^{\otimes l}) \\
&\quad\quad \left.\phantom{\sum_{k=2}^{n+2}\sum_{i+j=k}} - d \circ  H_{k-1} \circ (G^{\otimes (i-1)} \otimes \lambda_{n+2-k} \otimes F^{\otimes (j-1)}) \right\}.
\end{align*}
On the other hand, keeping in mind the previous writing, the condition that $\lambda \colon F \Rightarrow G$ be an $\A_{\infty}$-natural transformation yields
\begin{align*}
&\lambda_n \circ \left( \sum_{i=1}^{n-1} (-1)^{n-i-1} \id^{\otimes (n-i-1)} \otimes d \otimes \id^{\otimes i} \right) \\
& = m \circ (G \otimes \lambda_{n-1}) - (-1)^n m \circ (\lambda_{n-1} \otimes F) + \lambda_{n-1} \circ \left( \sum_{i=0}^{n-2} (-1)^{n-i} \id^{\otimes (n-i-2)} \otimes m \otimes \id^{\otimes i}\right) - d \circ \lambda_n.
\end{align*}
Using this in the terms of the second sum within the curly bracket gives, 
\begin{align*}
&H_1 \circ m \circ (G \otimes \lambda_{n-1}) - (-1)^{n} H_1 \circ m \circ (\lambda_{n-1} \otimes F) \\
&+ \sum_{k=3}^{n+2} \sum_{i+j=k} \Bigg\{\sum_{\substack{ l=0 \\ l \not\in\{i-1,\dots,n-j\}}}^{n-2} (-1)^{n-l-1}  H_{k-1} \circ (G^{\otimes (i-1)} \otimes \lambda_{n+2-k} \otimes F^{\otimes (j-1)}) \circ (\id^{\otimes (n-l-1)} \otimes d \otimes \id^{\otimes l})  \\
&\phantom{+\sum_{k=3}^{n+2}\sum_{i+j=k}} + H_{k-1}  \circ \left[ G^{\otimes (i-1)} \otimes ( m \circ (G \otimes \lambda_{n+1-k}) - (-1)^n m \circ (\lambda_{n+1-k} \otimes F) - d \circ \lambda_{n+2-k}) \otimes F^{(j-1)}\right] \\
& \phantom{+\sum_{k=3}^{n+2}\sum_{i+j=k}} + d \circ H_{k-1} \circ  (G^{\otimes (i-1)} \otimes \lambda_{n+2-k} \otimes F^{\otimes (j-1)})\Bigg\},
\end{align*}
and, after some reordering,
\begin{align*}
&\sum_{k=2}^{n+2} \sum_{i+j=k} \Bigg\{ \sum_{l=0}^{k-2} (-1)^{l} H_{k-1} \circ (\id^{\otimes (n-l-2)} \otimes d \otimes \id^{\otimes l}) \circ (G^{\otimes (i-1)} \otimes \lambda_{n+2-k} \otimes F^{\otimes (j-1)}) \\
&\phantom{\sum_{k=2}^{n+2} \sum_{i+j=k}}\quad + \sum_{l=0}^{k-3} (-1)^{l} H_{k-2}\circ (\id^{\otimes (n-l-3)} \otimes m \otimes \id^{\otimes l}) \circ (G^{\otimes (i-1)} \otimes \lambda_{n+1-k} \otimes F^{\otimes (j-1)}) \\
&\phantom{\sum_{k=2}^{n+2} \sum_{i+j=k}}\quad + d \circ H_{k-1} \circ (G^{\otimes (i-1)} \otimes \lambda_{n+2-k} \otimes F^{\otimes (j-1)}) \Bigg\}.
\end{align*}
Now, we use the fact that $H$ is a DG functor. This implies, in particular, that
$$
\sum_{p + q = k} H_{p}  \circ H_{q} =  \sum_{p+q+1=k} H_{k} \circ (\id^{\otimes p} \otimes d \otimes \id^{\otimes q}) +  \sum_{p+q+2=k} H_{k} \circ (\id^{\otimes p} \otimes m \otimes \id^{\otimes q}) + d \circ H_k.
$$
Plugging this back into the last expression above, we obtain
\begin{align*}
& m \circ \left\{ (m \circ (H \otimes G)) \otimes \left[ \sum_{k=2}^{n+1}  \sum_{i + j = k} H_{k-1} \circ (G^{\otimes (i-1)} \otimes \lambda_{n+1-k} \otimes F^{\otimes (j-1)})\right]   \right\} \\
& -(-1)^n m \circ \left\{  \left[ \sum_{k=2}^{n+1}  \sum_{i + j = k} H_{k-1} \circ (G^{\otimes (i-1)} \otimes \lambda_{n+1-k} \otimes F^{\otimes (j-1)})\right] \otimes (m \circ (H \otimes F))\right\},
\end{align*}
which, attending to the definitions, gives the desired result.

For the second part, we observe that $(H \circ \lambda)_0 = H_1 \circ \lambda_0$ and, by our hypothesis on $\lambda$, we know that $\lambda_0(A)$ is an isomorphism for all $A \in \Ob \Ccal$. Thus, for all $A \in \Ob \Ccal$,
$$
H_1(\lambda_0(A)) \circ H_1 (\lambda_0(A)^{-1}) = H_1(\lambda_0(A) \circ  \lambda_0(A)^{-1}) = H_1(\id_{A}) =   \id_{H_0(A)}.
$$
This implies that $(H \circ \lambda)_0(A)$ is an isomorphism for all $A \in \Ob \Ccal$, as was to be shown.
\end{proof}



\subsection{$\infty$-Local systems}\label{sec:2.3}
Let $E = \bigoplus_{k \in \ZZ} E^k$ be a graded vector bundle over a manifold $X$. We consider the space of $E$-valued differential forms $\Omega^{\sbullet}(X,E)$ to be graded with respect to the total degree. A \emph{superconnection} on $E$ is an operator $D \colon \Omega^{\sbullet} (X,E) \to \Omega^{\sbullet}(X,E)$ of degree $1$ which satisfies the Leibniz rule
$$
D (\sigma \wedge \omega)= d \sigma \wedge \omega + (-1)^k \sigma \wedge D \omega,
$$
for all $\sigma \in \Omega^k(X)$ and $\omega \in  \Omega^{\sbullet} (X,E)$. The \emph{curvature} of $D$ is the operator $D^2$. This is an $\Omega^{\sbullet}(X)$-linear operator on $\Omega^{\sbullet} (X,E)$ of degree $2$ which is given by multiplication by an element of $\Omega^{\sbullet} (X,\End(E))$. If $D^2=0$, then we say that $D$ is a \emph{flat} superconnection. By an \emph{$\infty$-local system} on $X$ we mean a graded vector bundle $E$ equipped with a flat superconnection $D$. We will denote such an $\infty$-local system by $(E,D)$. 

It is useful to spell out what the flatness condition means for a given $\infty$-local system $(E,D)$ on $X$. The Leibniz rule implies that $D$ is completely determined by its restriction to $\Omega^0(X,E)$. Then we may decompose
$$
D = \sum_{k \geq 0} D_k,
$$
where $D_k$ is of partial degree $k$ with respect to the $\ZZ$-grading on $\Omega^{\sbullet}(X)$. It is clear that each $D_k$ for $k \neq 1$ is $\Omega^{\sbullet}(X)$-linear and therefore it is given by multiplication by an element $\alpha_k \in \Omega^k(X,\End(E)^{1-k})$. On the contrary, $D_1$ satisfies the Leibniz rule on each of the vector bundles $E^k$, so it must be of the form $d_{\nabla}$, where $\nabla$ is an ordinary connection on $E$ which preserves the $\ZZ$-grading. We can thus write
$$
D = d_{\nabla} + \alpha_0 + \alpha_2 + \alpha_3 + \cdots.
$$
From this formula, it is straightforward to check that the condition $D^2 = 0$ amounts to
\begin{align*}
\alpha_0^2 &= 0, \\
d_{\nabla} \alpha_0 &=0, \\
[\alpha_0,\alpha_2] + F_{\nabla} &= 0, \\
[\alpha_0,\alpha_{n+1}] + d_{\nabla}\alpha_n + \sum_{k=2}^{n-1} \alpha_k \wedge \alpha_{n+1-k} &= 0, \,\, n \geq 2,
\end{align*}
where $F_{\nabla}$ is the curvature of the connection $\nabla$. The first identity implies that we have a cochain complex of vector bundles with differential $\alpha_0$. The second equation expresses the fact that $\alpha_0$ is covariantly constant with respect to the connection $\nabla$. The third equation indicates that the connection $\nabla$ fails to be flat up to terms involving the homotopy $\alpha_2$ and the differential $\alpha_0$.  

Let us assume that $E$ is trivialized over $M$. This means that $E = X \times V$ for some graded vector space $V = \bigoplus_{k \in \ZZ} V^{k}$. In this case, we have $\alpha_k \in \Omega^k(X, \End(V)^{1-k})$ for $k \neq 1$. Moreover, we can write $d_{\nabla} = d + \alpha_1$ for some $\alpha_1 \in \Omega^1(X, \End(V)^{0})$. Thus, the superconnection $D$ may be expressed as $D = d + \alpha$, where $\alpha \in \Omega^{\sbullet}(X,\End(V))$ is the homogeneous element of total degree $1$ defined by $\alpha = \sum_{k \geq 0} \alpha_k$. In addition, a straightforward calculation gives
$$
D^2  = d \alpha + \alpha \wedge \alpha.
$$
Consequently, the totality of equations of the flatness condition is equivalent to the single statement that $\alpha$ satisfies
$$
d \alpha + \alpha \wedge \alpha = 0.
$$
In the terminology of \S~\ref{sec:2.1}, this says that $\alpha$ is a Maurer-Cartan element of the DG Lie algebra $\Omega^{\sbullet}(M,\End(E))$. 

As mentioned in the introduction, $\infty$-local systems on a manifold $X$ can be naturally organized into a DG category, which we denote by $\Loc_{\infty}(X)$. Its objects are, of course, $\infty$-local systems $(E,D)$ on $X$. Given two $\infty$-local systems $(E,D)$ and $(E',D')$ we define the space of morphisms to be the graded vector space $\Omega^{\sbullet}(X,\Hom(E,E'))$ with the differential $\partial_{D,D'}$ acting as 
$$
\partial_{D,D'} \omega = D' \wedge \omega - (-1)^{k} \omega \wedge D,
$$
for any homogeneous element $\omega$ of degree $k$. If $(E,D)$ and $(E',D')$ are trivialized over $X$ as in the previous paragraph, then $\partial_{D,D'}$ may be expressed by
$$
\partial_{D,D'} \omega = d \omega + \alpha' \wedge \omega - (-1)^{k} \omega \wedge \alpha.
$$

Let us now briefly recall the pullback operation of $\infty$-local systems. For a smooth map $f \colon X \to Y$ between two manifolds $X$ and $Y$, there is a DG functor $f^{*} \colon \Loc_{\infty}(Y) \to \Loc_{\infty}(X)$ which sends $E$ with structure superconnection 
$$
D = d_{\nabla} + \alpha_0 + \alpha_2 + \alpha_3 + \cdots,
$$
to $f^* E$ endowed with 
$$
f^* D = d_{f^*\nabla} + f^* \alpha_0 + f^* \alpha_2 + f^* \alpha_3 - \cdots,
$$
where $f^*E$ is the pullback of $E$ and $f^* \nabla$ is the pullback connection on $f^{*} E$. One can easily check that $(f^*E,f^*D)$ is indeed an $\infty$-local system on $X$, so the DG functor $f^*$ is well defined. We refer to it as the \emph{pullback functor} induced by $f$.

We will use the following homotopy invariance result, which is Propositon~4.4 of \cite{CAA-AQV-SVV2019}. 

\begin{proposition}\label{prop:2.3}
Let $X$ be a manifold and denote by $\iota_{s} \colon X \to X \times [0,1]$ the inclusion at height $s$ given by $\iota_{s}(p) = (p,s)$. Then, there exists an $\A_{\infty}$-natural isomorphism $\lambda \colon \iota_0^* \Rightarrow \iota_1^*$ between the pullback functors $\iota_0^*,\iota_1^* \colon \Loc_{\infty}(X \times  [0,1]) \to \Loc_{\infty}(X)$. 
\end{proposition}

The construction of such $\A_{\infty}$-natural isomorphism is explicit and goes as follows. Pick a collection of $\infty$-local systems $(E_0,D_0),\dots,(E_n,D_n)$ on $M \times [0,1]$ so that each $(E_i,D_i)$ is trivialized over $M \times [0,1]$, that is, $E_i = (M \times [0,1]) \times V_i$ for some graded vector space $V_i = \bigoplus_{k \in \ZZ}V_i^{k}$ and $D_i = d - \alpha_i$ for some Maurer-Cartan element $\alpha_i \in \Omega^{\sbullet}(M \times [0,1],\End(V_i))$. Next pick homogeneous elements $\xi_0 \in \Omega^{\sbullet}(X \times [0,1], \Hom(V_0,V_1)),\dots,\xi_{n-1} \in \Omega^{\sbullet}(X \times [0,1], \Hom(V_{n-1},V_n))$. If we put $V = \bigoplus_{i=0}^{n} V_i$ then both the elements $\alpha_0,\dots,\alpha_n$ and the elements $\xi_0,\dots,\xi_{n-1}$ may be seen as elements of $\Omega^{\sbullet}(X \times [0,1], \End(V))$. Setting $\omega = \sum_{i=0}^{n} \alpha_i + \sum_{i=1}^{n-1} \xi_i$, one can then consider the iterated integral of $\omega$, which is defined as
$$
\Phi^{\omega} = \id_V + \sum_{k = 1}^{\infty} (-1)^{\vert \omega \vert \sum_{i=1}^{k-1} (k-i)} \int_{\Delta_k} \pi_1^* \omega \wedge \cdots \wedge \pi_k^* \omega,
$$
where $\Delta_k$ is the standard $k$-simplex and where, for each $i = 1,\dots,k$, $\pi_i \colon X \times \Delta_k \to X \times [0,1]$ denotes the natural projection defined by $\pi_i (x, (s_1,\dots,s_n)) = (x,s_i)$. It is straightforward to check that $\Phi^{\omega}$ determines an element of $\Omega^{\sbullet}(X, \End(V))$. This allows us to define $\lambda_n (\xi_{n-1} \otimes \cdots \otimes \xi_{0}) \in \Omega^{\sbullet}(M,\Hom(V_0,V_n))$ as the $(0,n)$ block entry of $\Phi^{\omega}$,  and in this way we end up with a linear map
$$
\lambda_n \colon \Omega^{\sbullet}(X \times [0,1], \Hom(V_{n-1},V_n)) \otimes \cdots \Omega^{\sbullet}(X \times [0,1], \Hom(V_{0},V_1)) \to \Omega^{\sbullet}(M,\Hom(V_0,V_n)).
$$
Finally, one proves that this map  has degree $-n$ and that it satisfies the required relations to be an $\A_{\infty}$-natural transformation. 

\subsection{The higher Riemann-Hilbert correspondence }\label{sec:2.4}
In this subsection we review the higher Riemann-Hilbert correspondence, which relates $\infty$-local systems on $X$ to representations up to homotopy of the smooth $\infty$-groupoid of $X$. Intuitively, the higher Riemann-Hilbert correspondence is the statement that, just as a flat connection can be integrated to a representation of the fundamental groupoid, a flat superconnection can be integrated to a representation of the $\infty$-groupoid.
The details of the proof can be found in \cite{Abad-Schatz2013,Block-Smith2014}.

Let $K_{\sbullet}$ be a simplicial set with face and degeneracy maps $d_i^{p} \colon K_{p+1} \to K_{p}$ and $s_i^{p}\colon K_{p-1} \to K_{p}$ for $0 \leq i \leq p$. For $p \leq n$, the maps that send an $n$-simplex to its front and back $p$-th face will be denoted by $f_p^{n} \colon K_{p} \to K_{n}$ and $b_p^{n} \colon K_{p} \to K_{n}$. Explicitly, $f_p^{n} = d_{n+1}^{n} \circ \cdots \circ d_{p}^{p-1}$ and $b_{p}^{n} = d_{0}^{n} \circ \cdots \circ d_{0}^{p-1}$. Also, for $x \in K_p$, we let $v_p(x)$ denote its $p$-th vertex. By a cochain of degree $p$ on $K_{\sbullet}$ with values in an algebra $A$ we mean a map $F \colon K_p \to A$. The cup product of two cochains $F$ and $F'$ of degrees $p$ and $p'$, respectively, is the cochain of degree $p+p'$ defined by the formula
$$
(F \abxcup F')(x) = F(f_{p+p'}^{p}(x)) F'(b_{p+p'}^{p'}(x)),
$$
for all $x \in K_{p + p'}$. 

With this background in mind, a \emph{representation up to homotopy} of $K_{\sbullet}$ is the datum of a graded vector bundle $E = \bigoplus_{k \in \ZZ} E^{k}$ over $K_0$ and a collection $F_{\sbullet}=\{F_p\}_{p \geq 0}$ where $F_p$ is a cochain of degree $p$ with $F_p(x) \in \Hom(E_{v_p(x)}, E_{v_0(x)})^{1-p}$ for $x \in K_p$, such that the relation
$$
\sum_{i=1}^{p-1} (-1)^{i} F_{p-1}(d_{i}^{p}(x)) + \sum_{i=0}^{p} (-1)^{i+1} (F_i \abxcup F_{p-i})(x) = 0
$$
is satisfied for any $p \geq 0$. We will denote such representation up to homotopy by $(E,F_{\sbullet})$. 

The meaning of the above relation becomes clear by looking at low degree instances. For $p = 0$, it states that we have a point $x \in K_0$ and a linear map $F_0(x) \colon E_x \to E_x$ of degree $1$ with the property $F_0 (x) \circ F_0 (x) = 0$.  This implies that we have a cochain complex
$$
\cdots \xrightarrow{F_0(x)} E_{x}^{k-1} \xrightarrow{F_0(x)}  E_{x}^{k} \xrightarrow{F_0(x)}  E_{x}^{k+1}\xrightarrow{F_0(x)} \cdots.
$$
For $p =1$, we have an element $x \in K_1$ and a linear map $F_1(x) \colon E_{v_1(x)} \to E_{v_0(x)}$ of degree $0$ such that $F_0(v_0(x)) \circ F_1(x) = F_1(x) \circ F_0(v_1(x))$. This implies that $F_1(x)$ defines a morphism of complexes. Finally, for $p =2$, we have an element $x \in K_2$ and a linear map $F_2(x) \colon E_{v_2(x)} \to E_{v_0(x)}$ of degree $-1$ such that
$$
F_1(f_{2}^{1}(x))\circ F_1(b_{2}^{1}(x)) - F_1(d_{1}^{2}(x)) = F_0(f_2^0(x)) \circ F_2(x) + F_2(x) \circ F_0(b_2^0(x)).
$$
This means that the cochain morphisms $F_1(f_{2}^{1}(x))\circ F_1(b_{2}^{1}(x))$ and $F_1(d_{1}^{2}(x))$ from $E_{v_2(x)}$ to $E_{v_0(x)}$ are homotopic via a homotopy given by $F_2(x)$. 

Representations up to homotopy of $K_{\sbullet}$ are organized in a DG category denoted $\Rep_{\infty}(K_{\sbullet})$. A morphism of degree $n$ between two representations up to homotopy $(E,F_{\sbullet})$ and $(E',F'_{\sbullet})$ is a collection $\varphi_{\sbullet} = \{\varphi_{p}\}_{p \geq 0}$ where $\varphi_p$ is a cochain of degree $p$ such that $\varphi_p(x) \in \Hom(E_{v_p(x)},E'_{v_0(x)})^{n-p}$ for $x \in K_p$. If $(E,F_{\sbullet})$, $(E',F'_{\sbullet})$ and $(E'',F''_{\sbullet})$ are three representations up to homotopy, and  $\varphi_{\sbullet} \colon (E,F_{\sbullet}) \to (E',F'_{\sbullet})$ and $\varphi'_{\sbullet} \colon (E',F'_{\sbullet}) \to (E'',F''_{\sbullet})$ are two morphisms of degree $n$ and $n'$, respectively, then $\varphi_{\sbullet}$ and $\varphi'_{\sbullet}$ can be composed as follows:
$$
(\varphi'_{\sbullet} \circ \varphi_{\sbullet})_p = \sum_{i = 0}^{p}  (-1)^{n (p-i)}(\varphi'_{p-i} \abxcup \varphi_{i} ).
$$
With respect to this composition, the identity morphism is the collection $\varphi_0 = \id_{K_0}$ and $\varphi_p = 0$ for $p \geq 1$. The set of morphisms of degree $n$ from $(E,F_{\sbullet})$ and $(E',F'_{\sbullet})$ is usually denoted by $\underline{\Hom}^n((E,F_{\sbullet}),(E',F'_{\sbullet}))$.  
The differential  is given by the formula
\begin{align*}
(\partial_{F,F'} \varphi_{\sbullet})_p (x) &= \sum_{i=0}^{p} (-1)^{n(p-i)} (F'_{p-i}  \abxcup \varphi_{i})(x) + \sum_{i=0}^{p} (-1)^{n + p -i +1} (\varphi_{p-i} \abxcup F_i)(x)  \\
&\quad \,+ \sum_{i=1}^{p-1} (-1)^{i+n} \varphi_{p-1}(d_i^{p-1}(x))
\end{align*}
for any homogeneous element $\varphi_{\sbullet}$ of degree $n$. 

We should remark that the DG category $\Rep_{\infty}(K_{\sbullet})$ is functorial with respect to simplicial maps. More precisely, for a simplicial map $f_{\sbullet} \colon K_{\sbullet} \to L_{\sbullet}$ between two simplicial sets $K_{\sbullet}$ and $L_{\sbullet}$, there is a DG functor $f_{\sbullet}^* \colon \Rep_{\infty}(L_{\sbullet}) \to \Rep_{\infty}(K_{\sbullet})$ which sends $(E,F_{\sbullet})$ to $(f_0^*E,F_{\sbullet}\circ f_{\sbullet})$. It is straightforward to check that the latter is indeed a representation up to homotopy of $K_{\sbullet}$, so that the DG functor $f_{\sbullet}^*$ is well defined. We call it the \emph{pullback functor} induced by $f_{\sbullet}$.  

Let $X$ be a smooth manifold. The simplicial set $\pi_{\infty} X_{\sbullet}$, called the \emph{smooth fundamental $\infty$-groupoid} of $X$, is defined by setting $\pi_{\infty} X_p$ to be the set of smooth maps from the standard $p$-simplex $\Delta_p$ to $X$. The simplicial maps are defined by pulling back along the cosimplicial maps between the simplices.  It turns out that the DG category  $\Rep_{\infty}(\pi_{\infty}X_{\sbullet})$ is a global version of the DG category $\Loc_{\infty}(X)$ of $\infty$-local systems on $X$. This is the content of the higher Riemann-Hilbert correspondence, which is the following result, proved in \cite{Block-Smith2014}.

\begin{theorem}\label{thm:2.4}
There exists an integration $\A_{\infty}$-functor
$$
\Ical \colon \Loc_{\infty}(X) \longrightarrow \Rep_{\infty}(\pi_{\infty} X_{\sbullet}), 
$$
which is an $\A_{\infty}$-quasi-equivalence of DG categories. 
\end{theorem}

\subsection{$\gfrak$-DG spaces and $\gfrak$-$L_{\infty}$ spaces}\label{sec:2.5}
Let $G$ be a connected Lie group with Lie algebra $\gfrak$. Consider the DG Lie algebra $\TT \gfrak$ defined as follows. As a vector space, $\TT \gfrak = \u\gfrak \oplus \gfrak$. For $x \in \gfrak$, we denote by $i(x) \in \TT \gfrak^{-1}$ and $L(x) \in \TT \gfrak^0$ the corresponding generators. The Lie bracket of $\TT\gfrak$ is given by the Cartan relations
\begin{align*}
[i(x),i(y)] &= 0,\\
 [L(x),L(y)] &= L([x,y]), \\
  [L(x),i(y)] &=i([x,y]).
\end{align*}
The differential is defined by
\begin{align*}
d (i(x)) &= L(x), \\
d (L(x) ) &= 0.
\end{align*}
By a $\gfrak$-\emph{DG space} we mean a cochain complex $V$ together with a DG Lie algebra homomorphism $\rho:\TT\gfrak \to \End(V)$. That is, it consists of a representation of $\TT\gfrak$ on $V$, where the operators $i_{x}\in \End(V)^{-1}$ and $L_{x} \in \End(V)^0$ correspond to $i(x)$ and $L(x) $, respectively. The operators $i_{x}$ are called \emph{contractions} and the operators $L_{x}$ are called \emph{Lie derivatives}. Given a $\gfrak$-DG space $V$, one defines the \emph{basic subspace} $V_{\bas}$ to be the cochain subcomplex consisting of all $v \in V$ with $i_x v = 0$ and $L_x v = 0$ for all $x \in \gfrak$. Equivalently, $V_{\bas}$ is the subspace fixed by the action of $\TT\gfrak$.

If $V$ and $W$ are $\gfrak$-DG spaces, a \emph{homomorphism} $f \colon V \to W$ is just a morphism of cochain complexes commuting with the operators $i_x$ and $L_x$. It is evident that such a homomorphism induces a morphism between the corresponding basic subspaces $V_{\bas}$ and $W_{\bas}$. 

We also need to consider $\gfrak$-DG algebras. A $\gfrak$-DG algebra is a $\gfrak$-DG space $A$ endowed with the structure of DG algebra such that the action of $\TT\gfrak$ is by derivations. Homomorphisms of $\gfrak$-DG algebras are homomorphism of $\gfrak$-DG spaces which are also homomorphisms of graded algebras. 

The canonical example of a $\gfrak$-DG algebra is the De Rham complex $\Omega^{\sbullet}(P)$ of a principal bundle $P$ over a smooth manifold $X$ with structure group $G$. Here the differential is the exterior derivative of forms $d_P$, and, if we let $\rho$ denote the infinitesimal action of the Lie algebra $\gfrak$ on $P$, $i_{x}$ is the inner product of a form with $\rho(x)$, and $L_{x}$ is the Lie derivative of the form along $\rho(x)$. 

Another example of a $\gfrak$-DG algebra is the Chevalley-Eilenberg algebra $\CE(\gfrak)$ of the Lie algebra $\gfrak$. 
Even though we already gave a general description of the Chevalley-Eilenberg algebra of a DG Lie algebra, we will also use the following, more explicit description in the case of a Lie algebra.
As a graded algebra it is the exterior algebra $\Lambda^{\sbullet}\gfrak^*$, where $\gfrak^*$ has degree $1$.  For $\xi \in \Lambda^1\gfrak^*$, $\dCE \xi$ is the element in $\Lambda^2\gfrak^*$ defined by
$$
(\dCE \xi)(x,y) = - \xi([x,y]),
$$
for all $x,y \in \gfrak$; $\dCE$ is then canonically extended to a derivation on $\Lambda^{\sbullet}\gfrak^*$. It follows from the Jacobi identity that $\dCE$ defined in this manner squares to zero. The derivations $i_{x}$ and $L_{x}$ are defined on generators $\xi \in \Lambda^1\gfrak^*$, by
\begin{align*}
i_{x} \xi &= \langle \xi, x \rangle, \\
L_{x} \xi &= \ad_{x}^* \xi,
\end{align*} 
where $\ad_{x}^*$ denotes the infinitesimal coadjoint action of the element $x$. Both are then canonically extended as derivations of degree $-1$ and $0$, respectively, to all of $\Lambda^{\sbullet}\gfrak^*$. 

Explicit formulas for these various maps, which will be useful later on, are obtained by introducing a basis for $\gfrak^*$. Let $e_a$ be a basis for $\gfrak$ with dual basis $e^{a}$ and structure constants $f^{a}_{\phantom{a}bc}= \langle e^{a},[e_b,e_c] \rangle$, and write $i_a$ and $L_a$ for the contraction $i_{e_a}$ and the Lie derivative $L_{e_a}$ acting on $\CE(\gfrak)$. Then the explicit formulas for $\dCE$, $i_{a}$ and $L_{a}$ are the following:
\begin{align*}
\dCE e^{a} &= - \frac{1}{2} f^{a}_{\phantom{a}bc} e^{b} \wedge e^{c}, \\
i_{b} e^a &= \delta_{b}^{\phantom{b}a}, \\
L_{b} e^{a} &= - f^{a}_{\phantom{a}bc}e^{c}. 
\end{align*} 
Here and throughout the text, the convention that repeated indices are summed over is in place.  

Given a commutative $\gfrak$-DG algebra $A$, an \emph{algebraic connection} is a map $\theta \colon \gfrak^* \to A^1$, which satisfies the relations
\begin{align*}
i_{x} (\theta(\xi)) &= \langle \xi,x \rangle, \\
L_{x} (\theta(\xi)) &= \theta(\ad_{x}^*\xi),
\end{align*}
for all $x \in \gfrak$ and $\xi \in \gfrak^*$. Given a principal bundle $P$ over a smooth manifold $X$ with structure group $G$ and setting $A = \Omega^{\sbullet}(P)$, this is equivalent to the usual definition of a connection on $P$.

The Weil algebra $\W\gfrak$, associated to the Lie algebra $\gfrak$, is the universal commutative $\gfrak$-DG algebra, with a connection $\iota \colon \gfrak^* \to \W^1\gfrak$. Thus, given a commutative $\gfrak$-DG algebra $A$, with connection $\theta$, there exists a unique $\gfrak$-DG algebra homomorphism $c_{\theta} \colon \W\gfrak \to A$ such that $c_{\theta}\circ \iota = \theta$. We will refer to $c_{\theta}$ as the \emph{characteristic homomorphism} for the connection $\theta$. 

The most explicit realisation of the Weil algebra $\W\gfrak$ is as follows. The underlying graded commutative algebra of $\W\gfrak$ is the tensor product
$$
\W \gfrak = \Lambda^{\sbullet}\gfrak^* \otimes \uS^{\sbullet}\gfrak^*,
$$
where $\uS^{\sbullet}\gfrak^*$ is the symmetric algebra of $\gfrak^*$ and where we associate to each $\xi \in \gfrak^*$ the degree $1$ generators $t(\xi) \in \Lambda^1 \gfrak^*$ and the degree $2$ generators $w(\xi) \in \uS^1\gfrak^*$. The differential on $\W\gfrak$ is characterised by the formulas
\begin{align*}
\dW (t(\xi)) &=  w(\xi) + \dCE (t(\xi)), \\
\dW (w(\xi)) &= \dCE (w(\xi)),
\end{align*}
where $\dCE$ is the differential of the Chevalley-Eilenberg complex $\CE(\gfrak)$ of $\gfrak$. The operators $i_x$ and $L_x$ are given on generators by 
\begin{align*}
i_x (t(\xi)) &= \langle t(\xi),x \rangle,\\
i_x (w(\xi)) &= 0, \\
L_x(t(\xi)) &= \ad_x^*(t(\xi)), \\
L_x(w(\xi)) &= \ad_x^*(w(\xi)),
\end{align*}
for $\xi \in \gfrak^*$, and extended uniquely as derivations. 

It will be useful to express the differential $\dW$ and the operator $i_x$ and $L_x$ in terms of a dual basis $e^{a}$ of $\gfrak^*$ and the structure constants $f^{a}_{\phantom{a}bc}$ of $\gfrak$. If we write $t^{a} = t(e^{a})$ and $w^{a} = w(e^{a})$, they are as follows:
\begin{align*}
\dW  t^{a} &= w^{a} - \frac{1}{2} f^{a}_{\phantom{a}bc} t^{b} t^{c}, \\
\dW w^{a} &= f^{a}_{\phantom{a}bc} w^{b} t^{c}, \\
i_b t^{a} &= \delta_{b}^{\phantom{b}a}, \\
i_b w^{a} &= 0, \\
L_{b} t^{a} &= - f^{a}_{\phantom{a}bc} t^{c}, \\
L_{b} w^{a} &=- f^{a}_{\phantom{a}bc} w^{c}. 
\end{align*}
It is clear that the Weil algebra is freely generated by $t^{a}$ and $\dW t^{a}$. This implies that $\W\gfrak$ is acyclic with respect to $\dW$. 

Finally, given a commutative $\gfrak$-DG algebra $A$ and a connection $\theta \colon \gfrak^* \to A^{1}$, the characteristic homomorphism $c_{\theta} \colon \W\gfrak \to A$ is defined on the generators of $\W\gfrak$, as follows:
\begin{align*}
c_{\theta} (t(\xi)) &= \theta(\xi) , \\
c_{\theta} (w(\xi)) &= \dA (\theta(\xi)) - \theta (\dCE(t(\xi))).
\end{align*}
Checking the definitions shows that $c_{\theta}$ is a chain map with respect to the differential $\dW$ of $\W\gfrak$.

For example, if we let $\theta$ be any connection on a principal bundle $P \to X$ with structure group $G$, then, by assigning to each $\xi \in \gfrak^*$ the form $\xi \circ \theta \in \Omega^{1}(P)$, we obtain a linear map $\gfrak^* \to \Omega^1(P)$; in view of the above, this map can be canonically extended to a $\gfrak$-DG homomorphism $c_{\theta} \colon \W\gfrak \to \Omega^{\sbullet}(P)$, which in turn induces a morphism of cochain complexes on the basic subspaces $c_{\theta} \colon (\W\gfrak)_{\bas} \to \Omega^{\sbullet}_{\bas}(P)$. As $\uS^{\sbullet}\gfrak^*$ is precisely the set of elements in $\W\gfrak$ killed by $i_x$ for $x \in \gfrak$, it follows that $(\W\gfrak)_{\bas}$ coincides with the algebra of invariant polynomials $(\uS^{\sbullet}\gfrak^*)_{\inv}$ on $\gfrak$. On the target complex we have on the other hand that $\Omega^{\sbullet}_{\bas}(P)$ is canonically isomorphic to $\Omega^{\sbullet}(X)$, so that in fact $c_{\theta} \colon (\uS^{\sbullet}\gfrak^*)_{\inv} \to \Omega^{\sbullet}(X)$. Since the differential $\dW$ vanishes on $(\W\gfrak)_{\bas}=(\uS^{\sbullet}\gfrak^*)_{\inv}$, it follows that $c_{\theta}$ induces a cohomology map $c_{\theta*} \colon (\uS^{\sbullet}\gfrak^*)_{\inv} \to \uH^{\sbullet}_{\DR}(X)$. This is the Chern-Weil homomorphism for the principal bundle $P \to X$.

We conclude this subsection by introducing a generalization of the notion of $\gfrak$-DG space which will play a key role in the sequel. By a \emph{$\gfrak$-$\L_{\infty}$ space} we mean a cochain complex $V$ together with an $L_{\infty}$-morphism $\Phi\colon \TT \gfrak \to \End(V)$. To contrast this notion with that $\gfrak$-DG space, recall from \S~\ref{sec:2.1} that such a morphism corresponds to a collection of linear maps $\Phi_n \colon \bigodot^n(\u \TT\gfrak) \to \u \End(V)$ of degree zero which satisfy the constraints coming from the condition that $\bar{\Phi}$ commutes with the codifferentials of $\bigodot^n(\u \TT\gfrak)$ and $\bigodot^n(\u \End(V))$. Upon setting $i_x = \Phi_1(i(x)) \in \End(V)^{-1}$ and $L_x = \Phi_1(L(x)) \in \End(V)^{0}$, for low values of $n$, the constraints read 
\begin{align*}
[i_x,\delta_V] &= L_x, \\
[L_x,\delta_V] &= 0, \\
[i_x,i_y]  &=  \Phi_2(i(x),L(y))- \Phi_2(L(x),i(y)) - \delta_V(\Phi_2(i(x),i(y))) , \\
 L_{[x,y]} - [L_x,L_y] &=  \delta_V(\Phi_2(L(x),L(y))), \\
 i_{[x,y]} - [L_x,i_y]  &=  \Phi_2(L(x),L(y)) + \delta_V(\Phi_2(L(x),i(y))).
\end{align*}
From these we gather that $\gfrak$-$\L_{\infty}$ spaces are generalizations of $\gfrak$-DG spaces where the higher maps $\Phi_n$ are homotopical corrections to the failure of the Cartan relations.  

It is clear from the definitions that the Weil DG algebra $\W\gfrak$ of $\gfrak$ is isomorphic to the Chevalley-Eilenberg DG algebra of $\TT\gfrak$. In view of Proposition~\ref{prop:2.1}, one concludes that a $\gfrak$-$\L_{\infty}$ space can be equivalently specified by a cochain complex $V$ together with a Maurer-Cartan element of $\W\gfrak \otimes \End(V)$. This fact will be used throughout the text. 

\section{The Chern-Weil construction for $\infty$-local systems}
In this section we prove one of our main results, Theorem A. We show that, given a principal bundle $\pi: P \to X$ with structure group $G$ and any connection $\theta$ on $P$, there is DG functor $\CW_{\theta} \colon \InfLoc_{\infty}(\gfrak) \to \Loc_{\infty}(X)$, where  $\InfLoc_{\infty}(\gfrak)$ is a the DG category of basic $\gfrak$-$\L_\infty$ spaces. Moreover, we show that, given a different connection $\theta'$, the functors $\CW_{\theta}$ and $\CW_{\theta'}$ are related by an $\A_\infty$-natural  isomorphism. This construction provides a categorification of the Chern-Weil homomorphism. 

\subsection{Basic $\infty$-local systems}\label{sec:3.1}
Assume that $\pi \colon P \to X$ is a principal bundle with structure group $G$ and $E= \bigoplus_{k \in \ZZ} E^{k}$ is a graded $G$-equivariant vector bundle on $P$.  By the latter we mean a graded vector bundle $E= \bigoplus_{k \in \ZZ} E^{k}$ on $P$ together with a right action of $G$ on $E$ that preserves the decomposition for which the projection from $E$ to $P$ is $G$-equivariant and $G$ acts linearly on the fibers. Such an action induces a right action of $G$ on $\Lambda^{\sbullet} T^*P \otimes E$ turning it into graded $G$-equivariant vector bundle on $P$. Thus we get a natural
 left action of $G$ on the space of $E$-valued differential forms $\Omega^{\sbullet}(P,E)$: if $\varphi \in \Omega^{r}(P,E)$ and $g \in G$, then $\varphi \cdot g$ is the element of $\Omega^{r}(P,E)$ whose value at any $p \in P$ and any $v_1,\dots,v_r \in T_p P$ is
$$
(g \cdot \varphi)_p(v_1,\dots,v_r) = \varphi_{p \cdot g}\big((d\sigma_{g})_p (v_1),\dots,(d\sigma_{g})_p(v_r)\big) \cdot g^{-1}.
$$
Let $\gfrak$ be the Lie algebra of $G$. For $x \in \gfrak$, we write $i_x^{E}$ for the contraction operator on $\Omega^{\sbullet}(P,E)$. We also denote by $L_x^{E}$ the corresponding infinitesimal action on 
$\Omega^{\sbullet}(P,E)$. An element $\varphi \in \Omega^{\sbullet}(P,E)$ is called \emph{basic} if it satisfies
\begin{align*}
i_x^{E} \varphi &= 0, \\
L_x^{E} \varphi &= 0,
\end{align*}
for all $x \in \gfrak$. Since each $i_x^{E}$ and $L_x^{E}$ are derivations, the basic elements are a graded subspace of $\Omega^{\sbullet}(P,E)$. This subspace will be denoted by $\Omega^{\sbullet}_{\bas}(P,E)$. 

A special case which will be important for us occurs when $E$ is trivialised over $P$ in such a way that $E = P \times V$ for some graded vector space $V = \bigoplus_{k \in \ZZ} V^{k}$ together with a representation $\rho$ of $G$ on $V$ that preserves the decomposition. In this case, $\Omega^{\sbullet}(P,E)$ coincides with the space of $V$-valued differential forms $\Omega^{\sbullet}(P,V) =\Omega^{\sbullet}(P) \otimes V$, and if for each $x \in \gfrak$, we write $i_x$ for  the contraction operator on $\Omega^{\sbullet}(P)$ and $L_x$ for both the contraction operator acting on $\Omega^{\sbullet}(P)$ and that acting on $V$, we have that $i_x^{E} = i_x \otimes 1$ and $L_x^{E} = L_x \otimes 1 + 1 \otimes L_x$. 

Next we consider the homomorphism $\pi^{*} \colon \Omega^{\sbullet}(X,E/G) \to \Omega^{\sbullet}(P,\pi^*(E/G))$ and the isomorphism $\Phi \colon \Omega^{\sbullet}(P,E) \to \Omega^{\sbullet}(P,\pi^*(E/G))$ induced by the natural isomorphism $E \to \pi^*(E/G)$. We define 
$$
\pi^{\#} \colon \Omega^{\sbullet}(X,E/G) \to \Omega^{\sbullet}(P,E)
$$
to be the composition $\pi^{\#} = \Phi^{-1} \circ \pi^{*}$. The following result is standard. 

\begin{proposition}\label{prop:3.1}
The homomorphism $\pi^{\#}$ is injective. The image of $\pi^{\#}$ consists precisely of the elements which are basic. 
\end{proposition}

This proposition shows that $\pi^{\#}$ can be considered as an isomorphism
$$
\pi^{\#} \colon \Omega^{\sbullet}(X,E/G) \xlongrightarrow{\cong} \Omega^{\sbullet}_{\bas}(P,E).
$$
If $E$ is trivialised over $P$ as in the previous paragraph, then
$$
\pi^{\#} \colon \Omega^{\sbullet}(X,P \times_{\rho} V) \xlongrightarrow{\cong}  \Omega^{\sbullet}_{\bas}(P,V),
$$
where $P \times_{\rho} V$ is the associated vector bundle determined by $\rho$. 

Now we come to the definition of basic superconnection. Let $\pi \colon P \to X$ be a principal bundle with structure group $G$ and let $E=\bigoplus_{k\in \ZZ} E^k$ be a graded $G$-equivariant vector bundle on $P$. A \emph{basic superconnection} on $P$ is a superconnection $D$ on $E$ which is $G$-equivariant and satisfies the property that
$$
[D,i_{x}^{E}] = L_{x}^{E},
$$
for all $x \in \gfrak$. The reason for this definition is made clear by the following result. 

\begin{lemma}\label{lem:3.2}
If $D$ is a basic superconnection on $E$, then $D$ preserves the graded subspace $\Omega^{\sbullet}_{\bas}(P,E)$. 
\end{lemma}

\begin{proof}
Since $D$ is a superconnection on $E$ which commutes with the action of $G$ on $\Omega^{\sbullet}(P,E)$, we see that 
$$
[D,L_{x}^{E}] = 0,
$$
for all $x \in \gfrak$. This, combined with the defining relation, implies that if $\varphi \in \Omega^{\sbullet}_{\bas}(P,E)$ then $D \varphi \in \Omega^{\sbullet}_{\bas}(P,E)$. 
\end{proof}

If $D$ is a basic superconnection on $E$, we will also denote its restriction to the graded subspace $\Omega^{\sbullet}_{\bas}(P,E)$ by $D$. By a \emph{basic $\infty$-local system} on $P$ we mean a graded $G$-equivariant vector bundle $E$ on $P$ endowed with a flat basic superconnection $D$. As usual, we will denote such a basic $\infty$-local system as a pair $(E,D)$. 

Just as with ordinary $\infty$-local systems, all basic $\infty$-local system on a principal bundle $\pi \colon P \to X$ can be naturally organised into a DG category, we denote by $[\Loc_{\infty}(P)]_{\bas}$. Its objects are, of course, basic $\infty$-local systems $(E,D)$ on $P$. Given two basic $\infty$-local systems $(E,D)$ and $(E',D')$ we define the space of morphism to be the graded vector space $\Omega^{\sbullet}_{\bas}(P,\Hom(E,E'))$ with the differential $\partial_{D,D'}$ acting as
$$
\partial_{D,D'} \varphi = D' \wedge \varphi - (-1)^k \varphi \wedge D,
$$ 
for any homogeneous basic element $\varphi$ of degree $k$. It is important to note that $[\Loc_{\infty}(P)]_{\bas}$ is a non full DG subcategory of $\Loc_{\infty}(P)$. 

We discuss next the geometric significance of $[\Loc_{\infty}(P)]_{\bas}$. For this purpose, consider the pullback DG functor $\pi^* \colon \Loc_{\infty}(X) \to \Loc_{\infty}(P)$. We have the following fundamental result.

\begin{proposition}\label{prop:3.3}
For every object $(E,D)$ in $[\Loc_{\infty}(P)]_{\bas}$ there is an isomorphism between $(E,D)$ and an object of the form $\pi^{*}(\bar{E},\bar{D})$ with $(\bar{E},\bar{D})$ in $\Loc_{\infty}(X)$. 
\end{proposition}

\begin{proof}
Let $(E,D)$ be an object in $[\Loc_{\infty}(P)]_{\bas}$. We know from Lemma~\ref{lem:3.2} that $D$ preserves the graded subspace $\Omega^{\sbullet}_{\bas}(P,E)$. Set $\bar{E}=E/G$ and define an operator $\bar{D}\colon \Omega^{\sbullet}(X,\bar{E}) \to \Omega^{\sbullet}(X,\bar{E})$ by means of the diagram 
$$
\xymatrix{ \Omega^{\sbullet}(X,\bar{E}) \ar[r]^-{\bar{D}} \ar[d]^-{\cong}_-{\pi^{\#}} & \Omega^{\sbullet}(X,\bar{E}) \ar[d]_-{\cong}^-{\pi^{\#}} \\ \Omega^{\sbullet}_{\bas}(P,E) \ar[r]^-{D} & \Omega^{\sbullet}_{\bas}(P,E).}
$$
Then, it is immediate to verify that $\bar{D}$ is a flat superconnection on $\bar{E}$. Thus, the pair $(\bar{E},\bar{D})$ defines an object in $\Loc_{\infty}(X)$. Now consider the isomorphism $\Phi \colon \Omega^{\sbullet}(P,E) \to \Omega^{\sbullet}(P,\pi^* \bar{E})$ defined as above. Since the isomorphism $E \to \pi^* \bar{E}$ is $G$-equivariant and the contraction operator $i_{x}^{E}$ only acts on $\Omega^{\sbullet}(P)$, it follows that $\Phi$ restricts to an isomorphism from $\Omega^{\sbullet}_{\bas}(P,E)$ to $\Omega^{\sbullet}_{\bas}(P,\pi^* \bar{E})$, which we also denote by $\Phi$. Notice that this isomorphism is of degree $0$. Moreover, bearing in mind the definition of $\bar{D}$, we obtain the commutative diagram
$$
\xymatrix{ \Omega^{\sbullet}_{\bas}(P,E) \ar[r]^-{D} \ar[d]_-{\Phi} & \Omega^{\sbullet}_{\bas}(P,E) \ar[d]^-{\Phi} \\ \Omega^{\sbullet}_{\bas}(P,\pi^* \bar{E}) \ar[r]^-{\pi^* \bar{D}} & \Omega^{\sbullet}_{\bas}(P,\pi^* \bar{E}).}
$$
This shows that $\Phi$ is an isomorphism from $(E,D)$ to $\pi^*( \bar{E},\bar{D})$.
\end{proof}

The previous proposition shows that $\pi^*$ can be considered as a DG functor
$$
\pi^* \colon \Loc_{\infty}(X) \to [\Loc_{\infty}(P)]_{\bas}.
$$
In fact, we obtain a lot more. 

\begin{proposition}\label{prop:3.4}
The DG functor $\pi^* \colon \Loc_{\infty}(X) \to [\Loc_{\infty}(P)]_{\bas}$ is a quasi-equivalence. 
\end{proposition}

\begin{proof}
By virtue of Proposition~\ref{prop:3.3}, we only have to show that $\pi^* \colon \Loc_{\infty}(X) \to [\Loc_{\infty}(P)]_{\bas}$ is quasi fully faithful. So for any pair of objects $(\bar{E},\bar{D})$ and $(\bar{E}',\bar{D}')$ in $\Loc_{\infty}(X)$ consider the associated map
$$
\pi^* \colon \Omega^{\sbullet}(X,\Hom(\bar{E},\bar{E}')) \longrightarrow \Omega^{\sbullet}_{\bas}(P,\Hom(\pi^*\bar{E},\pi^*\bar{E}')).
$$
Then, if $\Psi \colon \Omega^{\sbullet}(X,\Hom((\pi^*\bar{E})/G,(\pi^*\bar{E}')/G)) \to \Omega^{\sbullet}(X,\Hom(\bar{E},\bar{E}'))$ denotes the isomorphism induced by the natural isomorphisms $(\pi^*\bar{E})/G \to \bar{E}$ and $(\pi^*\bar{E}')/G \to \bar{E}'$, is not hard to check that $\pi^* \circ \Psi = \pi^{\#}$. This shows that $\pi^*$ is an isomorphism. On the other hand, if $\omega \in \Omega^{\sbullet}(X,\Hom(\bar{E},\bar{E}'))$ is a homogeneous element of degree $k$, we have
$$
\pi^*( \partial_{\bar{D},\bar{D}'} \omega) = \pi^*( \bar{D}' \wedge \omega - (-1)^{k} \omega \wedge \bar{D} ) = \pi^*\bar{D}' \wedge \pi^*\omega - (-1)^{k} \pi^*\omega \wedge \pi^*\bar{D} = \partial_{\pi^*\bar{D},\pi^*\bar{D}'} (\pi^* \omega ).
$$
It follows that $\pi^*$ is in fact an isomorphism of cochain complexes and thus, in particular, a quasi-isomorphism. 
\end{proof}

We shall see the importance of Proposition~\ref{prop:3.4} in the following section.

\subsection{The Chern-Weil DG functor}\label{sec:3.2}
In this part we describe the construction of a characteristic DG functor which extends the Chern-Weil homomorphism for principal bundles to the realm of $\infty$-local systems. We begin by recording a number of preliminary observations. 

Let $G$ be a connected Lie group with Lie algebra $\gfrak$ and let $V$ be a $\gfrak$-$\L_{\infty}$ space. For $x \in \gfrak$, by a slight abuse of notation, we will indistinctly write $i_x$ and $L_x$ for the contraction and Lie derivative operators acting on $\W\gfrak$ or $V$. With this caveat, it is a fact that $\W\gfrak \otimes V$ acquires the structure of a $\gfrak$-$\L_{\infty}$ space, where the differential, contraction and Lie derivative operators are $\dW \otimes 1 + 1 \otimes \dV$, $i_{x} \otimes 1$ and $L_{x}  \otimes 1 + 1 \otimes L_{x}$, respectively. Thus, we may consider the basic subspace $(\W\gfrak \otimes V)_{\bas}$. 

On the other hand, recall from the remark made at the end of \S\ref{sec:2.5} that the $\gfrak$-$\L_{\infty}$ space $V$ determines and is determined by a Maurer-Cartan element of $\W\gfrak \otimes \End(V)$, which we write as $\alpha_V$. We shall say that $V$ is \emph{basic} if the following identities are satisfied
\begin{align*}
[\alpha_V, i_{x}\otimes 1]  &=1 \otimes L_{x}, \\
[\alpha_V, L_{x}  \otimes 1 + 1 \otimes L_{x}] &= 0,
\end{align*}
for all $x \in \gfrak$. This definition is justified by the following construction.

We wish to construct a derivation $D  \colon \W\gfrak \otimes V \to \W\gfrak \otimes V$ of homogeneous degree $1$, such that $D^2=0$. To do so, we may simply define
$$
D = \dW \otimes 1 + 1 \otimes \dV + \alpha_V. 
$$
That $D^2$ is zero follows by follows by a straightforward calculation. Also, the following property holds true. 

\begin{lemma}\label{lem:3.5}
If $V$ is a basic $\gfrak$-$\L_{\infty}$ space, then $D$ preserves the graded subspace $(\W\gfrak \otimes V)_{\bas}$. 
\end{lemma}

\begin{proof}
We wish to show that if $\varphi \in (\W\gfrak \otimes V)_{\bas}$ then $D \varphi \in (\W\gfrak \otimes V)_{\bas}$.
For this it will be enough to verify that
$$
[D, L_c \otimes 1 + 1 \otimes L_c] = 0,
$$
and 
$$
[D,i_c \otimes 1] = L_c \otimes 1 + 1 \otimes L_c. 
$$
Fix an element of $\W\gfrak \otimes V$ of the form $\varphi \otimes v$. Then a straightforward computation gives
\begin{align*}
D\left( (L_{c} \otimes 1 + 1 \otimes L_{c})(\varphi \otimes v)\right) &= \dW(L_c \varphi) \otimes v + (-1)^{\vert \varphi \vert} L_{c}\varphi \otimes \dV v \\
&\quad +  \dW \varphi \otimes L_c v  + (-1)^{\vert \varphi \vert} \varphi \otimes \dV( L_c  v) \\
&\quad + \alpha_V \left( (L_{c} \otimes 1 + 1 \otimes L_{c})(\varphi \otimes v)\right) ,
\end{align*}
and 
\begin{align*}
(L_{c} \otimes 1 + 1 \otimes L_{c})\left( D(\varphi \otimes v)\right) &=  L_c(\dW \varphi) \otimes v + (-1)^{\vert \varphi \vert} L_{c}\varphi \otimes \dV v \\
&\quad +  \dW \varphi \otimes L_c v  + (-1)^{\vert \varphi \vert} \varphi \otimes L_c (\dV v) \\
&\quad + (L_{c} \otimes 1 + 1 \otimes L_{c})\left( \alpha_V(\varphi \otimes v)\right).
\end{align*}
Therefore,
\begin{align*}
[D, L_{c} \otimes 1 + 1 \otimes L_{c}] (\varphi \otimes v) &= [\dW,L_c]\varphi \otimes v + (-1)^{\vert \varphi \vert} \varphi \otimes [\dV,L_c]v \\
&\quad + [\alpha_V , L_{c} \otimes 1 + 1 \otimes L_{c}](\varphi \otimes v) \\
&= 0.
\end{align*}
Thus the first identity is established. On the other hand, again by a direct computation,
\begin{align*}
D \left( (i_c \otimes 1)(\varphi \otimes v) \right) &= \dW (i_c \varphi) \otimes v - (-1)^{\vert \varphi \vert} i_c \varphi \otimes \dV v  \\
&\quad + \alpha_V \left( (i_c \otimes 1)(\varphi \otimes v) \right),
\end{align*}
and
\begin{align*}
(i_c \otimes 1)\left( D(\varphi \otimes v)\right) &= i_c (\dW \varphi) \otimes v + (-1)^{\vert \varphi \vert} i_c \varphi \otimes \dV v  \\
&\quad  + (i_c \otimes 1)\left( \alpha_V (\varphi \otimes v)\right).
\end{align*}
Hence,
\begin{align*}
[D, i_{c} \otimes 1] (\varphi \otimes v) &= [\dW,i_c] \varphi \otimes v +  [\alpha_V, i_c \otimes 1](\varphi \otimes v) \\
&= L_c \varphi \otimes v + (1 \otimes L_c) (\varphi \otimes v) \\
&= (L_c \otimes 1 + 1 \otimes L_c) (\varphi \otimes v),
\end{align*}
and, consequently, the second identity also holds. 
\end{proof}

The content of the previous discussion is that, provided $V$ is a basic $\gfrak$-$\L_{\infty}$ space, the differential $D$ can be regarded as a flat basic superconnection on the graded vector bundle $EG \times V$. As a result, ignoring the technical problem with infinite dimensionality, the pair $(EG \times V, D)$ defines a basic $\infty$-local system on $EG$ in the sense of \S \ref{sec:3.1}. 

The preceding discussion allows us to define a DG category, which we call the \emph{DG category of infinitesimal $\infty$-local systems} on $\gfrak$, by the following data. The objects of this DG category are all basic $\gfrak$-$\L_{\infty}$ spaces. For any two $\gfrak$-$\L_{\infty}$ spaces $V$ and $V'$, with corresponding differentials $D$ and $D'$, the space of morphisms is the graded vector space $(\W\gfrak \otimes \Hom(V,V'))_{\bas}$ with the differential $\partial_{D,D'}$ acting according to the formula
$$
\partial_{D,D'}\varphi = D' \circ \varphi - (-1)^k \varphi \circ D,
$$
for any homogeneous element $\varphi$ of degree $k$. The DG category given by this data will be denoted by $\InfLoc_{\infty}(\gfrak)$. 

We are now in position to state and prove the main result of this section. 

\begin{theoremA}\label{thm:3.6}
Let $G$ be a Lie group and let $\pi \colon P \to X$ be a principal bundle with structure group $G$. Then, for any connection $\theta$ on $P$, there is a natural DG functor 
$$
\CW_{\theta} \colon \InfLoc_{\infty}(\gfrak) \longrightarrow \Loc_{\infty}(X).
$$
Moreover, for any two connections $\theta$ and $\theta'$ on $P$, there is an $A_{\infty}$-natural isomorphism between $\CW_{\theta}$ and $\CW_{\theta'}$. 
\end{theoremA}

\begin{proof}
In view of Proposition~\ref{prop:3.4}, the pullback DG functor $\pi^* \colon \Loc_{\infty}(X) \to [\Loc_{\infty}(P)]_{\bas}$ is a quasi-equivalence. Hence, it will suffice to show that for any connection $\theta$ on $P$, there is a natural DG functor 
$$
\overline{\CW}_{\theta} \colon \InfLoc_{\infty}(\gfrak) \longrightarrow [\Loc_{\infty}(P)]_{\bas}.
$$
Let us thus fix a connection $\theta$ on $P$ and keep the notation as above. We define the DG functor $\overline{\CW}_{\theta} \colon \InfLoc_{\infty}(\gfrak) \to [\Loc_{\infty}(P)]_{\bas}$ as follows. For each object $V$ in $\InfLoc_{\infty}(\gfrak)$ consider the operator $D_{\theta} \colon \Omega^{\sbullet}(P,V) \to \Omega^{\sbullet}(P,V)$ determined by the formula
$$
D_{\theta} = d_P \otimes 1 + 1 \otimes \dV + (c_{\theta} \otimes 1) \alpha_V,
$$
where as usual $c_{\theta} \colon \W\gfrak \to \Omega^{\sbullet}(P)$ is the characteristic homomorphism for the connection $\theta$. Here we note that the fact that $c_{\theta}$ is a $\gfrak$-DG algebra homomorphism implies that $(c_{\theta} \otimes 1) \alpha_V$ is a Maurer-Cartan element of $\Omega^{\sbullet}_{\bas}(P,\End(V))$. Then the argument given in the proof of Lemma\ref{lem:3.5} can be repeated to show that $D_{\theta}$ defines a flat basic superconnection on $P \times V$. That being the case, we define $\overline{\CW}_{\theta}(V)$ to be the basic $\infty$-local system $(P \times V , D_{\theta})$. On the other hand, suppose $\varphi \in (\W\gfrak \otimes \Hom(V,V'))_{\bas}$ is an arbitrary morphism between two objects $V$ and $V'$ in $\InfLoc_{\infty}(\gfrak)$. Then, by the foregoing remark, we again have that $(c_{\theta} \otimes 1) \varphi \in \Omega^{\sbullet}_{\bas}(P,\Hom(V,V'))$. In this way we get a morphism of graded vector spaces
$$
\overline{\CW}_{\theta} \colon (\W\gfrak \otimes \Hom(V,V'))_{\bas} \longrightarrow  \Omega^{\sbullet}_{\bas}(P,\Hom(V,V')). 
$$
We may note further that, since $c_{\theta}$ commutes with the differentials $\dW$ and $d_P$, this morphism is a cochain map with respect to both differentials $\partial_{D,D'}$ and $\partial_{D_{\theta},D'_{\theta}}$. It is a straightforward exercise to check that this data defines a DG functor $\overline{\CW}_{\theta} \colon \InfLoc_{\infty}(\gfrak) \to [\Loc_{\infty}(P)]_{\bas}$. 

To show the second part, suppose that $\theta$ and $\theta'$ are two different connections on $P$. Define an interpolating connection $\theta_t$ by 
$$
\theta_t = \theta + t(\theta - \theta')
$$
so that $\theta_0 = \theta$ and $\theta_1 = \theta'$. Then, if we let $\pr_1 \colon P \times [0,1] \to P$ be the projection onto the first factor, we have that $\hat{\theta}=\pr_1^* \theta_t$ defines a connection on $P \times [0,1]$. Hence, by the foregoing, we get a DG functor $\overline{\CW}_{\hat{\theta}} \colon \InfLoc_{\infty}(\gfrak) \to [\Loc_{\infty}(P \times [0,1])]_{\bas}$. If, on the other hand, $\iota_t \colon P \to P \times [0,1]$ denotes the inclusion of height $t$, then the pullback DG functor $\iota_t^* \colon \Loc_{\infty}(P \times [0,1]) \to \Loc_{\infty}(P)$ induces DG functors between the DG subcategories $[\Loc_{\infty}(P \times [0,1])]_{\bas}$ and $[\Loc_{\infty}(P)]_{\bas}$, which we denote by the same symbol, and further we have that $\overline{\CW}_{\theta} = \iota_0^* \circ \overline{\CW}_{\hat{\theta}}$ and $\overline{\CW}_{\theta'} = \iota_1^* \circ \overline{\CW}_{\hat{\theta}}$. In addition, by virtue of Proposition~\ref{prop:2.3}, there exists an $\A_{\infty}$-natural isomorphism between $\iota_0^*$ and $\iota_1^*$. By restricting the latter to the full DG subcategory of $[\Loc_{\infty}(P \times [0,1])]_{\bas}$ consisting of objects of the form $\overline{\CW}_{\hat{\theta}}(V)$ with $V$ an object of $\InfLoc_{\infty}(\gfrak)$, we obtain an $\A_{\infty}$-natural isomorphism between $\overline{\CW}_{\theta}, \overline{\CW}_{\theta'} \colon \InfLoc_{\infty}(\gfrak) \to [\Loc_{\infty}(P)]_{\bas}$, as wished. 
\end{proof}

We shall henceforth refer to the DG functor $\CW_{\theta} \colon \InfLoc_{\infty}(\gfrak) \to \Loc_{\infty}(X)$ induced by a connection $\theta$ on $P$ as the \emph{Chern-Weil DG functor of $P$}.

\subsection{An example:~the Gauss-Manin $\infty$-local system}
In the following we conserve the notations of \S\ref{sec:2.5}. Let $\gfrak$ be a Lie algebra and consider the DG Lie algebra $\W\gfrak \otimes \End(\CE(\gfrak))$. Fix a basis $e_a$ of $\gfrak$ with structure constants $f^{a}_{\phantom{a}bc}$ and recall that $t^{a}$ stands for the degree $1$ generators of $\Lambda^1 \gfrak^*$ and $w^{a}$ stands for the degree $2$ generators of $\uS^1\gfrak^*$. 

Our starting point is the following observation. 

\begin{lemma} \label{lem:3.7}
The element $\alpha_{\CE} = t^{a} \otimes L_{a} - w^{a} \otimes i_{a}$ is a Maurer-Cartan element of $\W\gfrak \otimes \End (\CE(\gfrak))$.  
\end{lemma}

\begin{proof}
On the one hand,
\begin{align*}
(\dW \otimes 1) \alpha_{\CE} &= (\dW \otimes 1) (t^{a} \otimes L_{a} - w^{a} \otimes i_{a} )\\
 &=  \dW t^{a} \otimes L_a - \dW w^{a} \otimes i_a \\
&=   w^{a} \otimes L_a - \frac{1}{2} f^{a}_{\phantom{a}bc} t^{b} t^{c} \otimes L_a - f^{a}_{\phantom{a}bc} w^{b} t^{c} \otimes i_a,
\end{align*}
and
\begin{align*}
(1 \otimes \dCE) \alpha_{\CE} &= (1 \otimes \dCE) ( t^{a} \otimes L_{a} - w^{a} \otimes i_{a} ) \\
&=- t^{a} \otimes [\dCE,L_a] - w^{a} \otimes [\dCE,i_a]  \\
&=- w^{a} \otimes L_a .
\end{align*}
Hence,
$$
(\dW \otimes 1 + 1 \otimes \dCE) \alpha_{\CE} =  -f^{a}_{\phantom{a}bc} w^{b} t^{c} \otimes i_a - \frac{1}{2} f^{a}_{\phantom{a}bc}t^{b} t^{c} \otimes L_a
$$
On the other hand, 
\begin{align*}
[\alpha_{\CE},\alpha_{\CE}] &= [t^{b} \otimes L_{b} -w^{b} \otimes i_{b}  ,  t^{c} \otimes L_{c}- w^{c} \otimes i_{c} ]  \\
&=   t^{b} t^{c} \otimes [L_b,L_c] - t^b w^c \otimes [L_b,i_c] + w^b t^c \otimes [i_b,L_c] + w^{b} w^{c} \otimes [i_b,i_c]   \\
&= f_{\phantom{a}bc}^{a} t^b t^c \otimes L_a - f_{\phantom{a}bc}^{a} t^b w^c \otimes i_a - f_{\phantom{a}cb}^{a} w^b t^c \otimes i_a  \\
&= f_{\phantom{a}bc}^{a} t^b t^c \otimes L_a + 2 f_{\phantom{a}bc}^{a} w^b t^c \otimes i_a.
\end{align*}
In conclusion, we obtain
$$
(\dW \otimes 1 + 1 \otimes \dCE) \alpha_{\CE} + \frac{1}{2}[\alpha_{\CE},\alpha_{\CE}] = 0 ,
$$
as required. 
\end{proof}

By the discussion of the preceding section, the Maurer-Cartan element $\alpha_{\CE}$ endows the Chevalley-Eilenberg algebra $\CE(\gfrak)$ with the structure of a $\gfrak$-$\L_{\infty}$ space. The next lemma deals with the basicness of such  $\gfrak$-$\L_{\infty}$ space.

\begin{lemma}
The $\gfrak$-$\L_{\infty}$ space $\CE(\gfrak)$ is basic. 
\end{lemma}  

\begin{proof}
It suffices to show that 
$$
[\alpha_{\CE}, i_{c} \otimes 1] = 1 \otimes L_{c},
$$
and 
$$
[\alpha_{\CE},L_c \otimes 1 + 1 \otimes L_{c}] = 0.
$$
So fix an element of $\W \gfrak \otimes \CE(\gfrak)$ of the form $\varphi \otimes \xi$. A simple calculation shows that
\begin{align*}
\alpha_{\CE}((i_c \otimes 1) (\varphi \otimes \xi)) = (t^{a} i_{c}\varphi) \otimes L_{a} \xi + (-1)^{\lvert \varphi\rvert} (w^{a} i_{c}\varphi) \otimes i_{a} \xi,
\end{align*}
and
\begin{align*}
(i_c \otimes 1) (\alpha_{\CE}(\varphi \otimes \xi)) = \delta^{a}_{c}\varphi \otimes L_{a}\xi - (t^{a}i_{c} \varphi) \otimes L_{a}\xi - (-1)^{\lvert \varphi\rvert} (w^{a} i_{c}\varphi) \otimes i_{a} \xi.
\end{align*}
Thus,
\begin{align*}
[\alpha_{\CE}, i_{c} \otimes 1] (\varphi \otimes \xi) = \delta^{a}_{c}\varphi \otimes L_{a}\xi = \varphi \otimes L_{c} \xi = (1 \otimes L_{c})(\varphi \otimes \xi).
\end{align*}
Hence the first identity holds. Furthermore, another straightforward calculation shows that
\begin{align*}
\alpha_{\CE}((L_c \otimes 1 + 1 \otimes L_c) (\varphi \otimes \xi)) &=  (t^{a} L_{c} \varphi) \otimes L_{a} \xi - (-1)^{\lvert \varphi\rvert} (w^{a} L_{c} \varphi) \otimes i_{a} \xi \\
&\quad  + (t^{a}\varphi) \otimes L_{a} (L_{c} \xi) - (-1)^{\lvert \varphi \rvert} (w^{a}\varphi) \otimes i_{a} (L_{c} \xi),
\end{align*}
and 
\begin{align*}
(L_c \otimes 1 + 1 \otimes L_c) (\alpha_{\CE} (\varphi \otimes \xi)) &= -(f^{a}_{\phantom{a}bc}t^{b} \varphi) \otimes L_{a} \xi + (t^{a} L_{c} \varphi) \otimes L_{a} \xi \\
&\quad + (-1)^{\lvert \varphi \rvert} (f^{a}_{\phantom{a}cb} w^{b} \varphi) \otimes i_{a} \xi - (-1)^{\lvert \varphi \rvert} (w^{a}L_{c}\varphi) \otimes i_{a} \xi \\
&\quad + (t^{a}\varphi) \otimes L_{c}(L_{a}\xi) - (-1)^{\lvert \varphi \rvert}(w^{a} \varphi) \otimes L_{c} (i_{a} \xi).
 \end{align*}
 Consequently, 
 \begin{align*}
 [\alpha_{\CE}, L_{c} \otimes 1 + 1 \otimes L_{c}] &= (f^{a}_{\phantom{a}cb}t^{b} \varphi) \otimes L_{a} \xi - (-1)^{\lvert \varphi \rvert} (f^{a}_{\phantom{a}cb}w^{b} \varphi) \otimes i_{a} \xi \\
 &\quad + (t^{a} \varphi) \otimes [L_a,L_c] \xi + (-1)^{\lvert \varphi \rvert} (w^{a}\varphi) \otimes [L_{c},i_{a}]\xi \\
 &= (f^{a}_{\phantom{a}cb}t^{b} \varphi) \otimes L_{a} \xi - (-1)^{\lvert \varphi \rvert} (f^{a}_{\phantom{a}cb}w^{b} \varphi) \otimes i_{a} \xi \\
 &\quad + (f^{b}_{\phantom{b}ac}t^{a} \varphi) \otimes L_{b} \xi + (-1)^{\lvert \varphi \rvert} (f^{b}_{\phantom{b}ca}w^{a} \varphi) \otimes i_{b} \xi \\
 &= 0,
 \end{align*}
 implying that the second identity also holds. 
 \end{proof}
 
 In view of this result, the $\gfrak$-$\L_{\infty}$ space $\CE(\gfrak)$ defines an object of the DG category $\InfLoc_{\infty}(\gfrak)$. The corresponding differential takes the form
 $$
 D = \dW \otimes 1 + 1 \otimes \dCE + \alpha_{\CE} = \dW \otimes 1 + 1 \otimes \dCE + t^{a} \otimes L_{a} - w^{a} \otimes i_{a}.
 $$
As pointed out in the previous section, in the special case in which $\gfrak$ is the Lie algebra of a connected Lie group, $D$ can be thought of as a flat superconnection on the graded vector bundle $EG \times \CE(\gfrak)$. Thus, ignoring the issue of infinite dimensionality, the pair $(EG \times \CE(\gfrak),D)$ defines an $\infty$-local system on $EG$, which we will call the \emph{universal Weil $\infty$-local system}. 

Let now $\pi \colon P \to X$ be a principal bundle with structure group $G$ and consider the coadjoint action $\ad^*$ of $G$ on $\gfrak^*$. This action extends uniquely to a representation of $G$ on $\CE(\gfrak)$, turning $P \times \CE(\gfrak)$ into a graded $G$-equivariant vector bundle on $P$. Choose a connection $\theta$ on $P$ with curvature $\Omega$ and write $\theta = \theta^{a} \otimes e_{a}$ and $\Omega = \Omega^{a} \otimes e_{a}$. By following the construction presented in the proof of Theorem~A, the operator $D_{\theta} \colon \Omega^{\sbullet}(P,\CE(\gfrak)) \to \Omega^{\sbullet}(P,\CE(\gfrak))$ given by
$$
D_{\theta} = d_P \otimes 1 + 1 \otimes \dCE + (c_{\theta} \otimes 1) \alpha_{\CE} = d_P \otimes 1 + 1 \otimes \delta_{CE} + \theta^{a} \otimes L_{a} - \Omega^{a} \otimes i_{a},
$$
defines a flat basic superconnection on $P \times \CE(\gfrak)$. In other words, we obtain a basic $\infty$-local system $(P \times \CE(\gfrak),D_{\theta})$, which is nothing but the image $\overline{\CW}_{\theta}(\CE(\gfrak))$ through the DG functor $\overline{\CW}_{\theta} \colon \InfLoc_{\infty}(\gfrak) \to [\Loc_{\infty}(P)]_{\bas}$ of $\CE(\gfrak)$. 

Next we consider the graded vector bundle $\CE(P) = P \times_{\ad^*} \CE(\gfrak)$ associated with $P$ and the coadjoint action $\ad^*$ of $G$ on $\CE(\gfrak)$. On account of Proposition~\ref{prop:3.3}, the flat basic superconnection $D_{\theta}$ defined above induces a flat superconnection $\bar{D}_{\theta}$ on $\CE(P)$. Thus, the pair $(\CE(P),\bar{D}_{\theta})$ constitutes an $\infty$-local system on the base manifold $X$, which again by the argument given in the proof of Theorem~A, is exactly the image of the image $\CW_{\theta}(\CE(\gfrak))$ under the Chern-Weil DG functor $\CW_{\theta} \colon \InfLoc_{\infty}(\gfrak) \to \Loc_{\infty}(X)$ of $\CE(\gfrak)$. We call it the \emph{Gauss-Manin $\infty$-local system for} $\theta$. 

We close with the following result.

\begin{proposition}
Let $G$ be a compact connected Lie group and let $\pi \colon P \to X$ be a principal bundle with structure group $G$. Then, for any connection $\theta$ on $P$, the cohomology of the corresponding Gauss-Manin $\infty$-local system $(\CE(P),\bar{D}_{\theta})$ is isomorphic to $\uH^{\sbullet}_{\DR}(P)$. 
\end{proposition}

\begin{proof}
Let us denote by $\Omega^{\sbullet}(P)^{G}$ the subspace of $G$-invariant elements of $\Omega^{\sbullet}(P)$, and write $\underline{\RR}$ for the constant $\infty$-local system $(X \times \RR,d)$. We will prove a stronger result, namely that the space of morphisms $\Hom(\underline{\RR}, (\CE(P),\bar{D}_{\theta}))$ is isomorphic as a cochain complex to $\Omega^{\sbullet}(P)^{G}$. That this implies our desired result is a consequence of the fact that the cohomology of $(\CE(P),\bar{D}_{\theta})$ isomorphic to that to that of $\Hom(\underline{\RR}, (\CE(P),\bar{D}_{\theta}))$, and that, since $G$ is compact and connected, the inclusion $\Omega^{\sbullet}(P)^{G} \to \Omega^{\sbullet}(P)$ is a quasi-isomorphism. By the preceding remark, we know that $(\CE(P),\bar{D}_{\theta})$ coincides with the image of $\CE(\gfrak)$ under the Chern-Weil DG  functor $\CW_{\theta} \colon \InfLoc_{\infty}(\gfrak) \to \Loc_{\infty}(X)$. Thus, by the construction of the latter, $\Hom(\underline{\RR}, (\CE(P),\bar{D}_{\theta}))$ is isomorphic to $(\Omega^{\sbullet}(P) \otimes \CE(\gfrak))_{\bas}$. Thus, it will be enough to show that $(\Omega^{\sbullet}(P) \otimes \CE(\gfrak))_{\bas}$ is isomorphic to $\Omega^{\sbullet}(P)^{G}$. For this, let us consider the subspace $(\Omega^{\sbullet}(P) \otimes \CE(\gfrak))_{\hor}$ consisting of all the elements of $\Omega^{\sbullet}(P) \otimes \CE(\gfrak)$ that are sent to zero by the contraction $i_x \otimes 1$ for all $x \in \gfrak$. Then the connection $\theta$ induces an isomorphism of cochain complexes between $(\Omega^{\sbullet}(P) \otimes \CE(\gfrak))_{\hor}$ and $\Omega^{\sbullet}(P)$. The result we are after follows immediately from the fact that this isomorphism is $G$-equivariant. 
\end{proof}


\section{$\infty$-Local systems and classifying spaces}
In this section we prove  our second main result, Theorem B. This states that if $G$ is a compact connected Lie group with Lie algebra $\gfrak$, then,  the DG category  $\InfLoc_{\infty}(\gfrak)$ is $\A_\infty$-quasi-equivalent to the DG category $\Loc_{\infty}(BG)$ of $\infty$-local systems on the classifying space of $G$. Moreover, given a principal bundle $\pi \colon P \to X$ with structure group $G$ and any connection $\theta$ on $P$, the Chern-Weil functor  $\CW_{\theta} \colon \InfLoc_{\infty}(\gfrak) \to \Loc_{\infty}(X)$ corresponds to the pullback functor associated to the classifying map of $P$.

\subsection{Canonical connections on Stiefel bundles}\label{sec:4.1}
Here we recall the existence of canonical connections on Stiefel bundles which will be needed in the text below. Although the results are well-known, proofs will be given for completeness sake. 

For nonnegative integers $k \leq n$, let $\V_k(\CC^n)$ denote the Stiefel manifold which parametrises orthonormal $k$-frames in $\CC^k$. It can be thought of as a set of $n \times k$ matrices by writing a $k$-frame as a matrix of $k$ column vectors in $\CC^n$.  The orthonormality condition is expressed by $A^* A = I_k$ where $A^*$ denotes the conjugate transpose of $A$ and $I_k$ denotes the $k \times k$ identity matrix. The topology on $\V_k(\CC^n)$ is the subspace topology inherited from $\CC^{n \times k}$. With this topology $\V_k(\CC^n)$ is a compact manifold whose dimension is $(2n-k)k$. There is a natural projection $\pi \colon \V_k(\CC^n) \to \G_k(\CC^n)$ from the Stiefel manifold $\V_k(\CC^n)$ to the Grassmannian of $k$-planes in $\CC^n$ which sends a $k$-frame to the subspace spanned by that frame. The fiber over a given point $W$ in $\G_k(\CC^n)$ is the set of all orthonormal $k$-frames contained in the space $W$. There is also a natural right action of $\mathrm{U}(k)$ on $\V_k(\CC^n)$ which rotates a $k$-frame in the space it spans. This action is free and the quotient is precisely $\G_k(\CC^n)$. It follows that $\pi \colon \V_k(\CC^n) \to \G_k(\CC^n)$ is indeed a principal bundle with structure group $\mathrm{U}(k)$.  

There are natural inclusions $i_n \colon \G_k(\CC^n) \to \G_k(\CC^{n+1})$ and $j_n\colon\V_k(\CC^{n}) \to \V_k(\CC^{n+1})$ induced by the standard inclusion $\CC^{n} \subset \CC^{n+1}$. Moreover, the latter is equivariant, so that we have a commutative diagram of inclusions and principal bundles with structure group $\mathrm{U}(k)$,
$$
\xymatrix{ \cdots \ar[r] & \V_k(\CC^{n-1}) \ar[r]^-{j_{n-1}} \ar[d]_-{\pi} & \V_k(\CC^{n}) \ar[r]^-{j_{n}} \ar[d]_-{\pi} & \V_k(\CC^{n+1}) \ar[r]  \ar[d]_-{\pi} & \cdots  \\ 
  \cdots \ar[r] & \G_k(\CC^{n-1}) \ar[r]^-{i_{n-1}} & \G_k(\CC^{n}) \ar[r]^-{i_{n}} & \G_k(\CC^{n+1}) \ar[r] & \cdots . }
$$
In the following we shall write simply $j$ to denote any of the inclusions in the upper line of the diagram.

\begin{proposition}\label{prop:4.1}
The principal bundles $\pi \colon \V_k(\CC^n) \to \G_k(\CC^n)$ admit a canonical connection $\omega_{\mathrm{c}}$ such that $j^* \omega_{\mathrm{c}} = \omega_{\mathrm{c}}$. 
\end{proposition}

\begin{proof}
We first identify $\V_k(\CC^n)$ with $n \times k$ matrices $A$ satisfying $A^* A= I_k$. Then the tangent space of $\V_k(\CC^n)$ at $A$ identified with the space of all $n \times k$ matrices $M$  such that
$$
M^* A + A^* M = 0.
$$
For each point $A$ in $\V_k(\CC^n)$ and for each tangent vector $M$ of $\V_k(\CC^n)$ at $A$, we consider the $k \times k$ matrix-valued $1$-form $\omega_{\mathrm{c}}$ on $\V_k(\CC^n)$ defined as
$$
[\omega_{\mathrm{c}}(A)] (M) = A^* M.
$$
Because of the above condition, it is clear that $\omega_{\mathrm{c}}$ takes actually values in the Lie algebra $\mathfrak{u}(k)$ of $\mathrm{U}(k)$. We claim that $\omega_{\mathrm{c}}$ is a connection form. In fact, if $u \in \mathrm{U}(k)$ and $r_u$ denotes the right translation, then for each point $A$ in $\V_k(\CC^n)$ and for each tangent vector $M$ of $\V_k(\CC^n)$ at $A$,
$$
[r_u^* \omega_{\mathrm{c}}(A)](M)=[\omega_{\mathrm{c}}(Au)] (Mu) = u^* A^* M u = u^{-1} A^* M u = \mathrm{Ad}_{u^{-1}} [\omega_{\mathrm{c}}(A)] (M).
$$
On the other hand, if $x \in \mathfrak{u}(k)$ and $x^{\sharp}$ denotes the fundamental vector field generated by $x$, then for each point $A$ in $\V_k(\CC^n)$,
$$
[\omega_{\mathrm{c}}(A)] (x^{\sharp}(A)) = [\omega_{\mathrm{c}}(A)] \left( \frac{d}{dt}\bigg\vert_{t=0} A \exp(t x)\right) = A^* A x = x. 
$$
Finally, the condition $j^* \omega_{\mathrm{c}} = \omega_{\mathrm{c}}$ follows immediately from the fact that $j$ is linear. 
\end{proof}

Now recall that every compact Lie group is isomorphic to a compact subgroup of $\mathrm{U}(k)$. Therefore, given such a Lie group $G$, one can fix an embedding $G \subset \mathrm{U}(k)$ and restrict the action of $\mathrm{U}(k)$ on $\V_k(\CC^n)$ to $G$. The quotient of $\V_k(\CC^n)$ by this action is a set $BG_n$ which can be given the structure of a smooth manifold. Thus we get a new principal bundle $\pi \colon \V_k(\CC^n) \to BG_{n}$ with structure group $G$. 

The inclusions $\CC^n \subset \CC^{n+1} \subset \cdots$ give inclusions $BG_n \to BG_{n+1} \to \cdots$ and taking the direct limit we arrive at the classifying space $BG$. Furthermore, as before, we have a commutative diagram of inclusions and principal bundles with structure group $G$,
$$
\xymatrix{ \cdots \ar[r] & \V_k(\CC^{n-1}) \ar[r]^-{j_{n-1}} \ar[d]_-{\pi} & \V_k(\CC^{n}) \ar[r]^-{j_{n}} \ar[d]_-{\pi} & \V_k(\CC^{n+1}) \ar[r]  \ar[d]_-{\pi} & \cdots  \\ 
  \cdots \ar[r] & BG_{n-1} \ar[r]^-{i_{n-1}} & BG_{n} \ar[r]^-{i_{n}} & BG_{n+1} \ar[r] & \cdots . }
$$
Keeping the notation introduced above, we have the following result.

\begin{proposition}\label{prop:4.2}
The principal bundles $\pi \colon \V_k(\CC^n) \to BG_n$ admit a canonical connection $\omega_G$ such that $j^* \omega_G = \omega_G$. 
\end{proposition}

\begin{proof}
Pick an $\mathrm{Ad}$-invariant inner product on the Lie algebra $\mathfrak{u}(k)$ of $\mathrm{U}(k)$. Then we have a direct sum decomposition
$$
\mathfrak{u}(k) = \mathfrak{g} \oplus \mathfrak{g}^{\perp},
$$ 
where $\mathfrak{g}$ is the Lie algebra of $G$ and $\mathfrak{g}^{\perp}$ is its orthogonal complement with respect to this inner product. We put
$$
\omega_G = \mathrm{pr}_{\mathfrak{g}} \circ \omega_{\mathrm{c}},
$$
where $\mathrm{pr}_{\mathfrak{g}}$ is the canonical projection onto the first summand of the above decomposition and $\omega_{\mathrm{c}}$ is the canonical connection of Proposition~\ref{prop:4.1}. We claim that $\omega_G$ satisfies the required conditions. Indeed, since the inner product on $\mathfrak{u}(k)$ is $\mathrm{Ad}$-invariant, we have that $\mathrm{pr}_{\mathfrak{g}} \circ \mathrm{Ad}_{u} = \mathrm{Ad}_{u} \circ \mathrm{pr}_{\mathfrak{g}}$ for all $u \in \mathrm{U}(k)$. Thus, if $u \in \mathrm{U}(k)$, then for each point $A$ in $\V_k(\CC^n)$ and for each tangent vector $M$ of $\V_k(\CC^n)$ at $A$,
\begin{align*}
[r_{u}^*\omega_G(A)](M) &= [\omega_G(Au)](Mu) = \mathrm{pr}_{\mathfrak{g}} \{ [\omega_{\mathrm{c}}(Au)](Mu) \} =\mathrm{pr}_{\mathfrak{g}} \{ \mathrm{Ad}_{u^{-1}} [\omega_{\mathrm{c}}(A)](M) \} \\
& = \mathrm{Ad}_{u^{-1}}  \mathrm{pr}_{\mathfrak{g}} \{ [\omega_{\mathrm{c}}(A)](M) \}=  \mathrm{Ad}_{u^{-1}} [\omega_G(A)](M). 
\end{align*}
On the other hand, if $x \in \mathfrak{g}$, then for each point $A$ in $\V_k(\CC^{n})$,
$$
[\omega_G (A)](x^{\sharp}(A)) = \mathrm{pr}_{\mathfrak{g}} \{ [\omega_{\mathrm{c}}(A)](x^{\sharp}(A)) \} = \mathrm{pr}_{\mathfrak{g}} x = x. 
$$
Finally, the condition $j^* \omega_G = \omega_G$ follows by a straightforward calculation. 
\end{proof}

\subsection{Smooth approximation for classifying spaces}\label{sec:4.2}
In this part we prove a smooth approximation result for the classifying space of a compact Lie group. We start by summarizing some of the necessary definitions.

Let $G$ be any compact Lie group, which we may suppose to be embedded in $\mathrm{U}(k)$ for some $k$, and let $BG$ denote its classifying space. As we have discussed in the previous section, $BG$ is the direct limit of finite-dimensional manifolds $BG_n$ with respect to natural inclusions. We may thus consider the smooth fundamental $\infty$-groupoid $\pi_{\infty}BG_{n\sbullet}$ of each $BG_n$, and define the smooth fundamental $\infty$-groupoid $\pi_{\infty}BG_{\sbullet}$ of $BG$ to be the direct limit of these simplicial sets over all $n$. On the other hand, if $BG$ is regarded as a topological space with the direct limit topology, we may also consider the continuous fundamental $\infty$-groupoid of $BG$, which we denote by $\pi^{\mathrm{c}}_{\infty}BG_{\sbullet}$. This is defined by setting $\pi^{\mathrm{c}}_{\infty}BG_{p}$ to be the set of continuous maps from the standard $p$-simplex $\Delta_p$ to $BG$. By the universal property of direct limits, there is a map of simplicial sets $j_{\sbullet}  \colon \pi_{\infty}BG_{\sbullet} \to \pi^{\mathrm{c}}_{\infty}BG_{\sbullet}$. The proof of the following result uses an argument we learned from Neil Strickland's answer to a question on MathOverflow (see \cite{StricklandMathoverflow}). 

\begin{proposition}\label{prop:4.3}
The map $j_{\sbullet} \colon \pi_{\infty}BG_{\sbullet} \to \pi^{\mathrm{c}}_{\infty}BG_{\sbullet}$ is a weak equivalence.  
\end{proposition}

\begin{proof}
Let $\vert \pi_{\infty}BG_{\sbullet} \vert$ and $\vert \pi^{\mathrm{c}}_{\infty}BG_{\sbullet} \vert$ be the geometric realisations of the simplicial sets $\pi_{\infty}BG_{\sbullet}$ and $\pi^{\mathrm{c}}_{\infty}BG_{\sbullet}$, and let $\vert j_{\sbullet} \vert \colon \vert \pi_{\infty}BG_{\sbullet} \vert \to \vert \pi^{\mathrm{c}}_{\infty}BG_{\sbullet} \vert$ be the continuous map induced by $j_{\sbullet}$. We must show that $\vert j_{\sbullet} \vert$ is a weak homotopy equivalence. To this end, we consider the evaluation maps $\mathrm{ev} \colon \vert \pi_{\infty}BG_{\sbullet} \vert \to BG$ and $\mathrm{ev}^{\mathrm{c}} \colon \vert \pi^{\mathrm{c}}_{\infty}BG_{\sbullet} \vert \to BG$. Then clearly there is a commutative diagram
$$
\xymatrix{\vert \pi_{\infty}BG_{\sbullet} \vert  \ar[rr]^-{\vert j_{\sbullet} \vert} \ar[dr]_-{\mathrm{ev}}& & \vert \pi^{\mathrm{c}}_{\infty}BG_{\sbullet} \vert \ar[ld]^-{\mathrm{ev}^{\mathrm{c}}} \\
 & BG. &}
$$
We know that $\mathrm{ev}^{\mathrm{c}}$ is a weak homotopy equivalence, therefore, it will be enough to show that $\mathrm{ev}$ is a weak homotopy equivalence. For this let $\eta \colon BG \to \vert \pi_{\infty}BG_{\sbullet} \vert$ be the natural map which sends $x \in BG$ to the equivalence class $[\sigma_x, \mathbf{0}]$, where $\sigma_x$ is the zero simplex mapping to the point $x$. It is immediate to check that $\mathrm{ev} \circ \eta = \id_{BG}$. On the other hand, consider the map $F\colon \vert \pi_{\infty}BG_{\sbullet} \vert \times [0,1] \to \vert \pi_{\infty}BG_{\sbullet} \vert$ defined for $\sigma \in \pi_{\infty}BG_p$, $\mathbf{t} \in \Delta_p$ and $\lambda \in [0,1]$ by
$$
F([\sigma,\mathbf{t}], \lambda) = [\sigma_{\mathbf{t},\lambda}, \mathbf{t}],
$$
where $\sigma_{\mathbf{t},\lambda} \in \pi_{\infty}BG_p$ here is given by $\sigma_{\mathbf{t},\lambda}(\mathbf{s}) = \sigma((1-\lambda)\mathbf{s} + \lambda \mathbf{t})$. Then $F$ is continuous and is clearly a homotopy between $\id_{BG}$ and $\eta \circ \mathrm{ev}$. This shows the desired assertion. 
\end{proof}

We also note the following general result by Holstein \cite[Corollary~20]{Holstein2015}.

\begin{lemma}\label{LHolstein}
Let  $f_{\sbullet}\colon K_{\sbullet }\to L_{\sbullet}$ be a weak equivalence of simplicial sets. Then the pull-back functor
$f_{\sbullet}^*\colon \Rep_{\infty}(L_{\sbullet}) \to \Rep_{\infty}(K_{\sbullet})$ is a quasi-equivalence.
\end{lemma}

Combining this lemma with Proposition \ref{prop:4.3}, we have the following. 

\begin{proposition}\label{prop:4.4}
The pullback functor $j_{\sbullet}^* \colon \Rep_{\infty}(\pi^{\mathrm{c}}_{\infty}BG_{\sbullet}) \to \Rep_{\infty}(\pi_{\infty}BG_{\sbullet})$ is a quasi-equivalence. 
\end{proposition} 

We conclude that the DG category $\Rep_{\infty}(\pi_{\infty}BG_{\sbullet})$ can be thought of as the category of $\infty$-local systems on $BG$. We will therefore abuse notation and write $\Loc_{\infty}(BG)$ instead of $\Rep_{\infty}(\pi_{\infty}BG_{\sbullet})$. 

\subsection{Auxiliary lemmas}
In this part we prove a number of technical results which will be crucial in our arguments later on, but which the reader only interested in the principal results might omit on first reading. We use freely the notation, terminology and definitions of \S\ref{sec:2.3}, \S\ref{sec:2.4} and \S\ref{sec:3.2}.

Let $X$ be a smooth manifold and let $(E,D)$ be an $\infty$-local system on $X$. We know already that $D$ can be expanded as
$$
D = d_{\nabla} + \alpha_0 + \alpha_2 + \alpha_3 + \cdots,
$$
where $\alpha_k \in \Omega^k(X,\End(E)^{1-k})$. Also, recall that $\alpha_0^2 = 0$. We will write $\Hcal^{\sbullet}(E)$ for the cohomology of the graded vector bundle $E$ with respect to $\alpha_0$. We also remark here that for a pair of objects $(E,D)$ and $(E',D')$ in $\Loc_{\infty}(X)$ a quasi-isomorphism from $(E,D)$ to $(E',D')$ is precisely a closed element of $\Omega^{\sbullet}(X,\Hom(E,E'))$ of degree $0$ for which its homogeneous component of partial degree $0$ is a quasi-isomorphism of cochain complexes. The following result is a direct translation of Theorem~4.13 in \cite{AriasAbad-Crainic2012}. 

\begin{lemma}\label{lem:4.5}
The complex $\Hcal^{\sbullet}(E)$ can be given the structure of a flat superconnection $\overline{D}$ such that there is an isomorphism of $\infty$-local systems from $(\Hcal^{\sbullet}(E),\overline{D})$ onto $(E,D)$. Moreover, there is a spectral sequence with second page $\Ecal_2^{p,q}=\uH^{p}(X,\Hcal^{q}(E))$ and which converges to $\uH^{p+q}(X,E)$. 
\end{lemma}

We consider the full DG subcategory $\Loc_{\infty}^{0}(X)$ of $\Loc_{\infty}(X)$ consisting of those $\infty$-local systems $(E,D)$ for which $\alpha_0 = 0$. The next result is an immediate consequence of the preceding lemma. 

\begin{lemma}\label{lem:4.6}
The natural inclusion DG functor $\Loc_{\infty}^{0}(X) \to \Loc_{\infty}(X)$ is a quasi-equivalence. 
\end{lemma}

Next, let us make an observation that will prove to be useful later. 

\begin{lemma}\label{lem:4.7}
Let $(E,D)$ and $(E',D')$ be two $\infty$-local systems on $X$. Then the space of morphism $\Omega^{\sbullet}(X,\Hom(E,E'))$ in the DG category $\Loc_{\infty}(X)$ carries a canonical decreasing filtration by partial degree. This filtration induces a spectral sequence with second page $\Ecal_2^{p,q} = \uH^{p}(X,\Hcal^{q}(\Hom(E,E')))$ and which converges to $\uH^{p+q}(\Omega^{\sbullet}(X,\Hom(E,E')))$. 
\end{lemma}

\begin{proof}
Explicitly, the filtration is
$$
F^p\Omega^{\sbullet}(X,\Hom(E,E')) = \bigoplus_{k\geq p}\Omega^{k}(X,\Hom(E,E')).
$$
The zeroth page of the spectral sequence is $\Ecal_0^{p,q} = \Omega^{p}(X,\Hom(E,E')^{q})$, where the differential is precisely the one induced by the differential on $\Hom(E,E')$. Therefore the first page of the spectral sequence is $\Ecal_1^{p,q} = \Omega^{p}(X,\Hcal^{q}(\Hom(E,E')))$, with first page differential given by the flat connection induced on $\Hcal^{\sbullet}(\Hom(E,E'))$. Thus, the second page of the spectral sequence becomes $\Ecal_2^{p,q} = \uH^{p}(X,\Hcal^{q}(\Hom(E,E')))$. 
\end{proof}

We also wish to consider a canonical filtration of the DG category $\Loc_{\infty}(X)$. If $m \in \NN$, we let $F^{m}\!\Loc_{\infty}(X)$ be the full DG subcategory of $\Loc_{\infty}(X)$ consisting of $\infty$-local systems $(E,D)$ for which $E^{k} = 0$ for all $k \in \ZZ$ such that $\vert k \vert > m$. Then it is obvious that $F^{m}\!\Loc_{\infty}(X) \subset F^{n}\!\Loc_{\infty}(X)$ whenever $m \leq n$, and any object of $\Loc_{\infty}(X)$ is contained in some $F^{m}\!\Loc_{\infty}(X)$. This means that $\Loc_{\infty}(X)$ is an $\NN$-filtered DG category.

Now let $Y$ be another smooth manifold and let $f \colon X \to Y$ be a smooth map. Consider the pullback DG functor $f^* \colon \Loc_{\infty}(Y) \to \Loc_{\infty}(X)$. Then it can easily be seen that, for any $m \in \NN$, $f^* (F^{m}\! \Loc_{\infty}(Y))$ is a full DG subcategory of $F^{m}\! \Loc_{\infty}(X)$. Thus $f^*$ is an $\NN$-filtered DG functor.

For the following discussion, let us write $\underline{\RR}$ for both the trivial vector bundle $X \times \RR$ of rank one and the constant $\infty$-local system $(X \times \RR, d)$. Also, let $(E,D)$ be an $\infty$-local system on $X$ and let $\xi$ be a closed element of $\Omega^{k}(X,\Hom(E,\underline{\RR})^{l-k})$ for some $k \geq 0$ and some $l \geq 0$. The  extension of $(E,D)$ by $\xi$, which we denote by $(E,D) \rtimes_{\xi} \underline{\RR}$, is the $\infty$-local system on $X$ whose underlying graded vector bundle is the direct sum $E \oplus \underline{\RR}[1-l]$ and whose flat graded superconnection is $D_{\xi} = D + \xi$.  

\begin{lemma}\label{lem:4.8}
Let $(E,D)$ and $(E',D')$ be two $\infty$-local systems on $X$ and let $\Phi$ be a quasi-isomorphism from $(E,D)$ onto $(E',D')$. If $\xi'$ is a closed element of $\Omega^{k}(X,\Hom(E',\underline{\RR})^{l-k})$ for some $k \geq 1$ and some $l \geq 0$, then $\Phi + \id_{\underline{\RR}[1-l]}$ defines a quasi-isomorphism from $(E,D) \rtimes_{\xi' \circ \Phi} \underline{\RR}$ onto $(E',D') \rtimes_{\xi'} \underline{\RR}$. 
\end{lemma}

\begin{proof}
Since, by definition, $\Phi$ is an element of $\Omega^{\sbullet}(X,\Hom(E,E'))$ of degree $0$, so is $\Phi + \id_{\underline{\RR}[1-l]}$ viewed as an element of $\Omega^{\sbullet}(X,\Hom(E \oplus \underline{\RR}[1-l],E' \oplus \underline{\RR}[1-l]))$. We must show that $\Phi + \id_{\underline{\RR}[1-l]}$ commutes with the graded superconnections $D_{\xi' \circ \Phi}$ and $D'_{\xi'}$. For this we simply compute:
\begin{align*}
D'_{\xi'} \circ (\Phi + \id_{\underline{\RR}[1-l]}) &= (D' + \xi') \circ (\Phi + \id_{\underline{\RR}[1-l]}) \\
&= D' \circ \Phi + \xi' \circ \Phi \\
&= \Phi \circ D + \xi' \circ \Phi \\
&= (\Phi + \id_{\underline{\RR}[1-l]}) \circ (D + \xi' \circ \Phi) \\
&= (\Phi + \id_{\underline{\RR}[1-l]}) \circ D_{\xi' \circ \Phi}. 
\end{align*}
Finally,  $\Phi + \id_{\underline{\RR}[1-l]}$ is a quasi-isomorphism because its homogeneous component of partial degree $0$ is  $\Phi_0 +  \id_{\underline{\RR}[1-l]}$.
\end{proof}

\begin{lemma}\label{lem:4.9}
Let $(E,D)$ be an $\infty$-local system on $X$. If $\xi$ is a closed element of $\Omega^{k}(X,\Hom(E,\underline{\RR})^{l-k})$ and $\eta$ is an element of $\Omega^{k}(X,\Hom(E,\underline{\RR})^{l-1-k})$ for some $k \geq 0$ and some $l \geq 0$, then $\id_{\underline{\RR}[1-l]} - \eta$ defines a quasi-isomorphism from $(E,D) \rtimes_{\xi} \underline{\RR}$ onto $(E,D) \rtimes_{\xi + D\eta} \underline{\RR}$. 
\end{lemma}

\begin{proof}
It is clear that $\id_{\underline{\RR}[1-l]} - \eta$ viewed as an element of $\Omega^{\sbullet}(X,\Hom(E \oplus \underline{\RR}[1-l],E' \oplus \underline{\RR}[1-l]))$ has degree $0$. We need to check that $\id_{\underline{\RR}[1-l]} - \eta$ commutes with the graded superconnections $D_{\xi}$ and $D_{\xi + D\eta}$. Again, we simply compute:
\begin{align*}
D_{\xi + D\eta} \circ (\id_{\underline{\RR}[1-l]} - \eta) &= (D + \xi + D\eta) \circ (\id_{\underline{\RR}[1-l]} - \eta) \\
&= D + \xi + D\eta - D \circ \eta \\
&= D + \xi - \eta \circ D \\
&= (\id_{\underline{\RR}[1-l]} - \eta) \circ (D + \xi) \\
&= (\id_{\underline{\RR}[1-l]} - \eta) \circ D_{\xi}.
\end{align*}
Finally, notice that $\id_{\underline{\RR}[1-l]} - \eta$ is invertible with inverse $\id_{\underline{\RR}[1-l]} + \eta$. 
\end{proof}

Let $K_{\sbullet}$ be a simplicial set and let $(E,F_{\sbullet})$ be a representation up to homotopy of $K_{\sbullet}$. In keeping with the notation used above, we will write $\Hcal^{\sbullet}(E)$ to represent the cohomology of the graded vector bundle $E$ with respect to $F_0$. We will also say that $(E,F_{\sbullet})$ is \emph{normalized} if each cochain $F_p$ vanishes on degenerate simplices. The following result is analogous to Lemma~\ref{lem:4.7}.

\begin{lemma}\label{lem:4.10}
Let $(E,F_{\sbullet})$ and $(E',F'_{\sbullet})$ be two normalized representations up to homotopy of $K_{\sbullet}$. Then the space of morphism $\underline{\Hom}^{\sbullet}((E,F_{\sbullet}),(E',F'_{\sbullet}))$ in the DG category $\Rep_{\infty}(K_{\sbullet})$ carries a decreasing filtration by cochain degree. This filtration induces a spectral sequence with second page $\Ecal_2^{p,q} = \uH^{p}(K_{\sbullet},\Hcal^{q}(\Hom(E,E')))$ and which converges to $\uH^{p+q}(\underline{\Hom}^{\sbullet}((E,F_{\sbullet}),(E',F'_{\sbullet})))$. 
\end{lemma}

\begin{proof}
The zeroth page of the spectral sequence $\Ecal^{p,q}_0$ is given by the space of singular $p$-cochains on $K_{\sbullet}$ with values in $\Hom(E,E')^{q}$ with the differential induced by that of $\Hom(E,E')$. Thus the first page of the spectral sequence $\Ecal^{p,q}_1$ is the space $p$-cochains on $K_{\sbullet}$ with values in $\Hcal^{q}(\Hom(E,E'))$, with the first page differential given by the induced one from $\Hcal^{\sbullet}(\Hom(E,E'))$. Hence, the second page of the spectral sequence is indeed $\Ecal_2^{p,q} = \uH^{p}(K_{\sbullet},\Hcal^{q}(\Hom(E,E')))$.
\end{proof}

Next we wish to consider a canonical filtration of $\Rep_{\infty}(K_{\sbullet})$, just as we did for the DG category of $\infty$-local systems. Thus, if $m \in \NN$, we set $F^{m} \! \Rep_{\infty}(K_{\sbullet})$ to be the full DG subcategory of $\Rep_{\infty}(K_{\sbullet})$ consisting of all representations up to homotopy $(E,F_{\sbullet})$ for which $E^{k} = 0$ for all $k \in \ZZ$ such that $\vert k \vert > m$. This turns $\Rep_{\infty}(K_{\sbullet})$ into an $\NN$-filtered DG category. Moreover, if $L_{\sbullet}$ is another simplicial set and $f_{\sbullet} \colon K_{\sbullet} \to L_{\sbullet}$ is a simplicial map, then, for any $m \in \NN$, the pullback DG functor $f_{\sbullet}^* \colon \Rep_{\infty}(L_{\sbullet}) \to \Rep_{\infty}(K_{\sbullet})$ sends $F^{m} \! \Rep_{\infty}(L_{\sbullet})$ to a full DG subcategory of $F^{m} \! \Rep_{\infty}(K_{\sbullet})$. In other words, $f_{\sbullet}^*$ is an $\NN$-filtered DG functor. 

Before moving on, we introduce some terminology and notation regarding simplicial sets. Denoting the usual category of finite ordinals by $\bm{\Delta}$, we obtain for each $n \geq 0$, a subcategory $\bm{\Delta}_{\leq n}$ determined by the objects $[k]$ of $\bm{\Delta}$ with $k \leq n$. A $n$-truncated simplicial set is a contravariant functor from $\bm{\Delta}_{\leq n}$ to the category of sets $\Set$. We denote the category of simplicial sets by $\sSet$ and the category of $n$-truncated simplicial sets by $\sSet_{\leq n}$. Restriction gives a truncation functor  $\tr_n \colon \sSet \to \sSet_{\leq n}$ which admits a right adjoint $\cosk_n \colon \sSet_{\leq n} \to \sSet$ called the $n$-coskeleton functor. By abuse of language, one refers to the composite functor $\Cosk_n = \cosk_n \circ \tr_n$ also as the $n$-coskeleton functor. It can be shown that, given an object $K_{\sbullet}$ in $\sSet$, the unit of the adjunction $\eta_{K} \colon K_{\sbullet} \to \Cosk_n(K_{\sbullet})$ induces a bijection on simplices of dimension $k \leq n$. Furthermore, if $K_{\sbullet}$ and $L_{\sbullet}$ are two objects in $\sSet$ which in addition are Kan complexes, and $f_{\sbullet} \colon K_{\sbullet} \to L_{\sbullet}$ is a morphism that induces an isomorphism in homotopy groups up to degree $k < n$, then the morphism $\Cosk_n(f_{\sbullet}) \colon \Cosk_n(K_{\sbullet}) \to  \Cosk_n(L_{\sbullet})$ is a weak equivalence. For details of these constructions, we refer the reader to \cite{Goerss-Jardine2009,Joyal-Tierney2008,Ruschoff2017}.

\begin{lemma}\label{lem:4.11}
Let $K_{\sbullet}$ and $L_{\sbullet}$ be simplicial sets and let $f_{\sbullet} \colon K_{\sbullet} \to L_{\sbullet}$ be a simplicial map that induces an isomorphism in homotopy groups up to degree $2m$. Then the pullback functor $f_{\sbullet}^* \colon \Rep_{\infty}(L_{\sbullet}) \to \Rep_{\infty}(K_{\sbullet})$ induces a quasi-equivalence between $F^{m}\! \Rep_{\infty}(L_{\sbullet})$ and $F^{m}\!\Rep_{\infty}(K_{\sbullet})$. 
\end{lemma}

\begin{proof}
The proof makes use of Kan's fibrant replacement functor $\Ex^{\infty} \colon \sSet \to \sSet$ (see for instance \S3 of \cite{Jardine1987}).  There are natural maps $\nu_{K} \colon K_{\sbullet} \to \Ex^{\infty}(K_{\sbullet})$ and $\nu_{L} \colon L_{\sbullet} \to \Ex^{\infty}(L_{\sbullet})$ which are weak equivalences. Moreover, as the name suggests, $\Ex^{\infty}(K_{\sbullet})$ and $\Ex^{\infty}(L_{\sbullet})$ are Kan complexes. Consider the following commutative diagram
$$
\xymatrix@C=16ex{K_{\sbullet} \ar[r]^-{f_{\sbullet}} \ar[d]_-{\nu_K}& L_{\sbullet} \ar[d]^-{\nu_L} \\
\Ex^{\infty}(K_{\sbullet})  \ar[d]_-{\eta_{\Ex^{\infty}(K)}}& \Ex^{\infty}(L_{\sbullet}) \ar[d]^-{\eta_{\Ex^{\infty}(L)} }\\
\Cosk_{2m+1}(\Ex^{\infty}(K_{\sbullet})) \ar[r]^{\Cosk_{2m+1}(\Ex^{\infty}(f_{\sbullet}))} & \Cosk_{2m+1}(\Ex^{\infty}(L_{\sbullet})).}
$$
In order to show that $f_{\sbullet}^*$ induces an equivalence on homotopy categories, it suffices to show that the pullback functor associated to all the other maps in the diagram have the corresponding property. Using Lemma~\ref{LHolstein}, we observe that, being weak equivalences, the pullback functors associated to $\nu_K$ and $\nu_L$ induce quasi-equivalences between $\Rep_{\infty}(\Ex^{\infty}(K_{\sbullet}))$ and $\Rep_{\infty}(K_{\sbullet})$,  and between $\Rep_{\infty}(\Ex^{\infty}(L_{\sbullet}))$ and $\Rep_{\infty}(L_{\sbullet})$, respectively. On the other hand, in view of our previous remarks, we know that the maps $\eta_{\Ex^{\infty}(K)}$ and $\eta_{\Ex^{\infty}(L)}$ are bijections on simplices of dimension $k \leq 2m+1$ and therefore their associated pullback functors induce equivalences between the homotopy categories of $F^{m}\!\Rep_{\infty}(\Cosk_{2m+1}(\Ex^{\infty}(K_{\sbullet})))$ and $F^{m}\! \Rep_{\infty}(\Ex^{\infty}(K_{\sbullet}))$,  and those of $F^{m}\!\Rep_{\infty}(\Cosk_{2m+1}(\Ex^{\infty}(L_{\sbullet})))$ and $F^{m}\! \Rep_{\infty}(\Ex^{\infty}(L_{\sbullet}))$, respectively. Besides, by our hypothesis on $f_{\sbullet}$, the map $\Ex^{\infty}(f_{\sbullet})$ induces an isomorphism in homotopy groups up to degree $2m$. One concludes that $\Cosk_{2m+1}(\Ex^{\infty}(f_{\sbullet}))$ is a weak equivalence, using again Lemma \ref{LHolstein}, one concludes that its associated pullback functor induces an equivalence between the homotopy categories of $F^{m}\!\Rep_{\infty}(\Cosk_{2m+1}(\Ex^{\infty}(K_{\sbullet})))$ and $F^{m}\!\Rep_{\infty}(\Cosk_{2m+1}(\Ex^{\infty}(L_{\sbullet})))$. This completes the proof of the lemma.
\end{proof}

Now let $G$ be a compact and simply connected Lie group with Lie algebra $\gfrak$. If $V$ is a $\gfrak$-$\L_{\infty}$ space, we shall write $\Hcal^{\sbullet}(V)$ for the cohomology of the graded vector space $V$ with respect to the zeroth component of its corresponding Maurer-Cartan element $\alpha_V$. A result analogous to Lemmas \ref{lem:4.7} and \ref{lem:4.10} is as follows.

\begin{lemma}\label{lem:4.12}
Let $V$ and $V'$ be two objects in $\InfLoc_{\infty}(\gfrak)$. Then the space of morphisms $(\W\gfrak \otimes \Hom(V,V'))_{\bas}$ carries a canonical decreasing filtration by the degree in the Weil algebra $\W\gfrak$. This filtration induces a spectral sequence with second page $\Ecal_2^{p,q} = \uH^{p}((\W\gfrak)_{\bas}) \otimes \Hcal^{q}(\Hom(V,V'))$ and which converges to $\uH^{p+q}((\W\gfrak \otimes \Hom(V,V'))_{\bas})$. 
\end{lemma}

\begin{proof}
The zeroth page of the spectral sequence is
$$
\Ecal_0^{p,q} = (\W^p\gfrak \otimes \Hom(V,V')^q)_{\bas} = (\uS^{p/2}\gfrak^* \otimes \Hom(V,V')^q)_{\inv}
$$ 
with differential induced by that of $\Hom(V,V')$. Since $G$ is compact, we know that taking cohomology commutes with taking invariants. Moreover, since $V$ and $V'$ are basic $\gfrak$-$\L_\infty$-spaces, the action of $\gfrak$ on $\Hcal^{\bullet}(\Hom(V,V'))$ is trivial. Therefore,
$$
\Ecal_1^{p,q} = (\uS^{p/2}\gfrak^*)_{\inv} \otimes \Hcal^{q}(\Hom(V,V')) = \uH^{p}((\W\gfrak)_{\bas}) \otimes \Hcal^{q}(\Hom(V,V')).
$$
As for the first page differential, it vanishes for degree reasons. We therefore conclude that the second page of the spectral sequence is $\Ecal_2^{p,q} =  \uH^{p}((\W\gfrak)_{\bas}) \otimes \Hcal^{q}(\Hom(V,V'))$, as asserted.
\end{proof}

It should also be apparent that, just as for the DG categories of $\infty$-local systems and representations up to homotopy, there is a canonical filtration $F^{m}\!\InfLoc_{\infty}(\gfrak)$ of $\InfLoc_{\infty}(\gfrak)$ consisting of all $\gfrak$-$\L_{\infty}$ spaces $V$ for which $V^{k} = 0$ for all $k \in \ZZ$ such that $\vert k \vert > m$. This makes $\InfLoc_{\infty}(\gfrak)$ into a filtered DG category. 

There is one more result that will be useful as we go forward.

\begin{lemma}\label{lem:4.13}
Let $G$ be a compact connected Lie group, $\pi \colon P \to X$ be a principal $G$-bundle and $\theta$  a connection on $P$. Then, for any pair of objects $V$ and $V'$ of $\InfLoc_{\infty}(\gfrak)$, the morphism of cochain complexes
$$
\CW_{\theta} \colon (\W\gfrak \otimes \Hom(V,V'))_{\bas} \longrightarrow \Omega^{\sbullet}(X,\Hom(\CW_{\theta}(V),\CW_{\theta}(V')))
$$
determined by the Chern-Weil DG functor of $P$, is compatible with the filtrations described in Lemmas \ref{lem:4.7} and \ref{lem:4.12}, and induces a morphism of spectral sequences. Moreover, if $X$ is simply connected, this morphism is identified with the Chern-Weil homomorphism of $P$ tensored with the identity on $\Hcal^{\sbullet}(\Hom(V,V'))$. 
\end{lemma}

\begin{proof}
The first part follows immediately from the definition of the Chern-Weil DG functor of $P$. To prove the second part we note that, since $G$ acts trivially on cohomology, the graded vector bundle $\Hcal^{\sbullet}(\Hom(\CW_{\theta}(V),\CW_{\theta}(V')))$ is trivial, so if we look at the first pages of the associated spectral sequences, we get the morphism
$$
\CW_{\theta} \colon \uH^{p}((\W\gfrak)_{\bas}) \otimes \Hcal^{q}(\Hom(V,V')) \longrightarrow \Omega^{p}(X,\Hcal^{q}(\Hom(V,V'))),
$$
and this is simply the tensor product of the characteristic homomorphism of $P$ with the identity on $\Hcal^{q}(\Hom(V,V'))$. Furthermore, by the hypothesis that $X$ is simply connected, we have that the graded vector bundle $\Hcal^{\sbullet}(\Hom(\CW_{\theta}(V),\CW_{\theta}(V')))$ is trivialised by the flat connection induced on it, so that the induced morphism between the second pages of the spectral sequences is
$$
\CW_{\theta} \colon \uH^{p}((\W\gfrak)_{\bas}) \otimes \Hcal^{q}(\Hom(V,V')) \longrightarrow \uH^{p}_{\DR}(X) \otimes \Hcal^{q}(\Hom(V,V')),
$$
and it is given by the tensor product of the the Chern-Weil homomorphism of $P$ with the identity on $\Hcal^{q}(\Hom(V,V'))$, as claimed.
\end{proof}

We close by mentioning that a construction analogous to that of an extension of an $\infty$-local system can be carried out for basic $\gfrak$-$\L_{\infty}$ spaces. Let us write $\RR$ for the trivial representation of $\TT\gfrak$. We further let $V$ be a basic $\gfrak$-$\L_{\infty}$ space and $\gamma$ be a closed element of $(\W\gfrak \otimes \Hom(V,\RR))_{\bas}$ of degree $l$. Then the extension of $V$ by $\gamma$, denoted by $V \rtimes_\gamma \RR$, is the basic $\gfrak$-$\L_{\infty}$ space whose underlying graded vector space is the direct sum $V \oplus \RR[1-l]$ and whose corresponding Maurer-Cartan element is $\alpha_{V} + \gamma$. 


\subsection{The Weil $\A_{\infty}$-functor for classifying spaces}
In this subsection we prove Theorem~B, which provides an infinitesimal description of the category of $\infty$-local systems on the classifying space $BG$ of a compact connected Lie group $G$ with Lie algebra $\gfrak$. First, let us recall our situation. 

We consider a compact and simply connected Lie group $G$ with Lie algebra $\gfrak$. For nonnegative integers $k \leq n$, we let $\pi \colon \V_k(\CC^n) \to BG_n$ be the principal bundle with structure group $G$ described in \S\ref{sec:4.1}. The canonical connection on $\V_k(\CC^n)$ constructed in Proposition~\ref{prop:4.2} we denote simply by $\omega$  in this discussion. We recall that it satisfies $j^* \omega = \omega$, where $j \colon \V_k(\CC^{n}) \to \V_k(\CC^{n+1})$ is the natural inclusion. We also write $\varphi_n \colon BG_n \to BG$ for the canonical map obtained from the direct limit construction. 

We will now proceed to construct the $\A_{\infty}$-functor that connects the DG categories $\InfLoc_{\infty}(\gfrak)$ and $\Loc_{\infty}(BG)$. By Theorem~A, we know that for each $n \geq k$, there is a Chern-Weil DG functor $\CW_{\omega}^{(n)} \colon \InfLoc_{\infty}(\gfrak) \to \Loc_{\infty}(BG_n)$. Because of the naturallity of the connection $\omega$, these DG functors fit into a commutative diagram
$$
\xymatrix{\Loc_{\infty}(BG_{n}) & & \Loc_{\infty}(BG_{n+1}) \ar[ll]_-{i^*}  \\  & \InfLoc_{\infty}(\gfrak) \ar[lu]^-{\CW_{\omega}^{(n)}}\ar[ru]_-{\CW_{\omega}^{(n+1)}}}
$$
where $i \colon BG_{n} \to BG_{n+1}$ is the natural inclusion. On the other hand, by the higher Riemann-Hilbert correspondence, for each $n$, there is an integration $\A_{\infty}$-functor $\Ical^{(n)} \colon \Loc_{\infty}(BG_n) \to \Rep_{\infty}(\pi_{\infty} BG_{n \sbullet})$, which is in addition an $\A_{\infty}$-quasi-equivalence. On the other hand, invoking higher Riemann-Hilbert correspondence, for each $n$, there is an integration $\A_{\infty}$-functor $\Ical^{(n)} \colon \Loc_{\infty}(BG_n) \to \Rep_{\infty}(\pi_{\infty} BG_{n \sbullet})$, which is in addition an $\A_{\infty}$-quasi-equivalence. By the naturality of the construction, all these $\A_{\infty}$-functors fit into the following commutative diagram:
$$
\xymatrix{ \Rep_{\infty}(\pi_{\infty} BG_{n \sbullet}) & \Rep_{\infty}(\pi_{\infty} BG_{n+1\sbullet}). \ar[l]_-{i_{\sbullet}^*} \\
\Loc_{\infty}(BG_{n}) \ar[u]^-{\Ical^{(n)}} & \Loc_{\infty}(BG_{n+1}) \ar[u]_-{\Ical^{(n+1)}} \ar[l]_-{i^*}}
$$
For each $n \geq 0$, we let $\Wcal^{(n)} \colon \InfLoc_{\infty}(\gfrak) \to \Rep_{\infty}(\pi_{\infty} BG_{n \sbullet})$ be the $\A_{\infty}$-functor defined as the composition
$$
\Wcal^{(n)} = \Ical^{(n)} \circ \CW_{\omega}^{(n)}.
$$
It follows from the above discussion that these $\A_{\infty}$-functors fit into a commutative diagram 
$$
\xymatrix{\Rep_{\infty}(\pi_{\infty}BG_{n \sbullet}) & & \Rep_{\infty}(\pi_{\infty}BG_{n+1 \sbullet}) \ar[ll]_-{i_{\sbullet}^*}  \\  & \InfLoc_{\infty}(\gfrak). \ar[lu]^-{\Wcal^{(n)}}\ar[ru]_-{\Wcal^{(n+1)}}}
$$
Taking note of the definition of the DG category $\Loc_{\infty}(BG)$ given at the end of \S\ref{sec:4.2}, and recalling that, for each $n \geq 0$, integrating an $\infty$-local system on $BG_n$ amounts to assigning holonomies to smooth maps from the standard simplices to $BG_n$, we deduce the existence of an $\A_{\infty}$-functor
$$
\Wcal \colon \InfLoc_{\infty}(\gfrak) \longrightarrow \Loc_{\infty}(BG).
$$ 
Our goal is to show that $\Wcal$ is in fact an $\A_{\infty}$-quasi-equivalence. The following two preliminary results will clear our path. 

\begin{proposition}\label{prop:4.14}
Let  $V$ and $V'$ be two objects in $F^{m} \! \InfLoc_{\infty}(\gfrak)$ and put $N = \frac{1}{2}(2k + l + 6m)$. Then, for each $n \geq N$, the morphism of cochain complexes
$$
\CW_{\omega}^{(n)} \colon (\W\gfrak \otimes  \Hom(V,V'))_{\bas} \longrightarrow \Omega^{\sbullet}(BG_n,\Hom(\CW_{\omega}^{(n)}(V), \CW_{\omega}^{(n)}(V')))
$$
induces an isomorphism in cohomology up to degree $l$. 
\end{proposition}

\begin{proof}
To start with, since $n \geq N$, it follows that $\V_{k}(\CC^{n})$ is $2(N-k)$-connected, see for instance \cite{Hatcher}  page 382. As a consequence, employing the long exact homotopy sequence, we deduce that the canonical map $\varphi_n \colon BG_n \to BG$ induces an isomorphism in homotopy groups up to degree $2(N-k)$. This in turn implies that the characteristic map $c_{\omega}\colon (\W\gfrak)_{\bas} \to \Omega^{\sbullet}(BG_n)$ induces an isomorphism in cohomology up to degree $2(N-k)$. On the other hand, by virtue of Lemma~\ref{lem:4.13}, the morphism of cochain complexes $\CW_{\omega}^{(n)}$ is compatible with the canonical decreasing filtrations on $(\W\gfrak \otimes  \Hom(V,V'))_{\bas}$ and $\Omega^{\sbullet}(BG_n,\Hom(\CW_{\omega}^{(n)}(V), \CW_{\omega}^{(n)}(V')))$, and induces a morphism between the corresponding spectral sequences, whose $r$th pages we write respectively as $\Ecal^{p,q}_r$ and $\Ecal'^{p,q}_r$. Moreover, since $BG_n$ is simply connected, this morphism is identified with the Chern-Weil homomorphism of $\V_k(\CC^n)$ tensored with the identity of $\Hcal^{\sbullet}(\Hom(V,V'))$. These facts, together with Lemmas~\ref{lem:4.7} and \ref{lem:4.12}, imply that the induced morphism between the second pages of the spectral sequences 
$$
\CW_{\omega}^{(n)} \colon \Ecal^{p,q}_2 \longrightarrow \Ecal'^{p,q}_2
$$ 
is an isomorphism for all $p \leq 2(N-k)$. But, by our hypothesis on $V$ and $V'$, it results that $\Hcal^{q}(\Hom(V,V')) = 0$ if $q + 2m < 0$. Thus, the above induced morphism is an isomorphism if $p+q \leq 2(N-k)-2m$. From this, by  induction on $r$, we infer that the induced morphism between the $r$th pages of the spectral sequences 
$$
\CW_{\omega}^{(n)} \colon \Ecal^{p,q}_r \longrightarrow \Ecal'^{p,q}_r
$$ 
is an isomorphism if $p+q \leq 2(N-k)-2m - r + 2$. In particular, if we take $r = 4m + 2$, then the induced morphism
$$
\CW_{\omega}^{(n)} \colon \Ecal^{p,q}_{4m+2} \longrightarrow \Ecal'^{p,q}_{4m+2}
$$
is an isomorphism if $p+q \leq 2(N-k)-6m = l$. At the same time, under this assumption on $p+q$, the $r$th page differentials vanish for all $r \geq 4m + 2$. We therefore conclude that the induced morphism between the limiting spectral sequences
$$
\CW_{\omega}^{(n)} \colon \Ecal^{p,q}_{\infty} \longrightarrow \Ecal'^{p,q}_{\infty}
$$
is an isomorphism if $p+q \leq l$. Appealing again to Lemmas~\ref{lem:4.7} and \ref{lem:4.12}, we get the desired conclusion.
\end{proof}

\begin{proposition}\label{prop:4.15}
Given $m \geq 0$, there exists a sufficiently large integer $n > 0$ such that the corresponding Chern-Weil DG funtor $\CW_{\omega}^{(n)} \colon \InfLoc_{\infty}(\gfrak) \to \Loc_{\infty}(BG_n)$ induces a quasi essentially surjective DG functor between $F^{m} \! \InfLoc_{\infty}(\gfrak)$ and $F^{m} \! \Loc_{\infty}(BG_n)$.
\end{proposition}

\begin{proof}
Pick $n$ sufficiently large so that $\V_k(\CC^{n})$ is $2m$-connected and $BG_n$ is simply connected. We must show that for any object $(E,D)$ of $F^{m} \! \Loc_{\infty}(BG_n)$, an object $V$ of $F^{m} \! \InfLoc_{\infty}(\gfrak)$ can be  found such that $(E,D)$ is quasi-isomorphic to $\CW_{\omega}^{(n)}(V)$. Thus, let us fix such an object $(E,D)$. Notice that, by Lemma~\ref{lem:4.6}, we may assume that $(E,D)$ is in fact an object  of $F^{m} \! \Loc_{\infty}^{0}(BG_n)$. Moreover, since $BG_n$ is simply connected, we may further assume that $(E,D)$ is trivialized over $BG_n$, that is to say, $E = BG_n \times W$ for some graded vector space $W = \bigoplus_{i = p}^{q} W^{i}$ with $-m \leq p \leq q \leq m$, and $D = d + \alpha$ for some Maurer-Cartan element $\alpha \in \Omega^{\sbullet}(BG_n, \End(W))$. We will therefore argue by induction on the dimension of $W$. If this dimension is equal to $0$, the result clearly holds. The result also holds if $p = q$, since in this case $(E,D)$ is isomorphic to the constant $\infty$-local system $(BG_n \times \RR, d)$. Suppose, then, that $p < q$ and that the result is true for all objects of $F^{m} \! \Loc_{\infty}^{0}(BG_n)$ arising from a graded vector space of dimension less than the dimension of $W$. Pick a nonzero vector in $W^{p}$ and consider the one-dimensional subspace $U$ spanned by this vector. We also fix an inner product on $W^{p}$ and write $U^{\perp}$ for the orthogonal complement of $U$ in $W^{p}$. Set $W' = U^{\perp} \oplus \left( \bigoplus_{i = p+1}^{q} W^{i}\right)$ so that $W = U \oplus W'$. Then the Maurer-Cartan element $\alpha$ can be decomposed as $\alpha = \alpha' + \xi$, where $\alpha' \in \Omega^{\sbullet}(BG_n, \End(W'))$ and $\xi \in \Omega^{\sbullet}(BG_n,\Hom(W',U))$ are homogeneous elements of total degree $1$. In terms of this decomposition, the Maurer-Cartan equation for $\alpha$ becomes 
$$
0 = d \alpha + \alpha \wedge \alpha = d \alpha' + d \xi + \alpha' \wedge \alpha' + \xi \wedge \alpha',
$$
which can be decoupled into two independent equations
\begin{align*}
d \alpha' + \alpha' \wedge \alpha' &= 0,  \\
d \xi + \xi \wedge \alpha' &= 0.
\end{align*}
Hence, if we put $E' = BG_n \times W'$ and $D' = d + \alpha'$, the first of these equations implies that $(E',D')$ defines an object of $F^{m}\! \Loc_{\infty}(BG_n)$. And since the dimension of $W'$ is less than the dimension of $W$, our induction hypothesis ensures the existence of an object $V'$ of $F^{m}\!  \InfLoc_{\infty}(\gfrak)$ together with a quasi-isomorphism $\Phi$ from $\CW_{\theta}^{(n)}(V')$ onto $(E',D')$. In addition to this, it is clear from the construction above that there is an isomorphism of $\infty$-local systems from $(E,D)$ onto the central extension $(E',D') \rtimes_{\xi} \underline{\RR}$. Therefore, by applying Lemma~\ref{lem:4.8}, we determine the existence of an isomorphism of $\infty$-local systems from $(E,D)$ onto $\CW_{\omega}^{(n)}(V') \rtimes_{\xi \circ \Phi} \underline{\RR}$. On the other hand, because $V'$ is an object of $F^{m}\!  \InfLoc_{\infty}(\gfrak)$, by virtue of Proposition~\ref{prop:4.14}, we know that, if $n \geq \frac{1}{2}(2k + 1 - p + 6m)$, the morphism of cochain complexes
$$
\CW_{\omega}^{(n)} \colon (\W\gfrak \otimes \Hom(V',\RR))_{\bas} \longrightarrow \Omega^{\sbullet}(BG_n, \Hom(\CW_{\omega}^{(n)}(V'), \underline{\RR}))
$$
induces an isomorphism in cohomology up to degree $1-p$. This means, in particular, that we can find a closed element $\gamma$ of $(\W\gfrak \otimes \Hom(V',\RR))_{\bas}$ of degree $1$ such that
$$
\CW_{\omega}^{(n)}(\gamma) = \xi \circ \Phi + D\eta,
$$
for some homogeneous element $\eta$ of $\Omega^{\sbullet}(BG_n, \Hom(\CW_{\omega}^{(n)}(V'), \underline{\RR}))$ of degree $0$. Therefore, combining the foregoing with Lemma~\ref{lem:4.9}, we obtain an isomorphism of $\infty$-local systems from $(E,D)$ onto $\CW_{\omega}^{(n)}(V') \rtimes_{\CW_{\omega}^{(n)}(\gamma)} \underline{\RR}$. But it is not hard to check that $\CW_{\omega}^{(n)}(V') \rtimes_{\CW_{\omega}^{(n)}(\gamma)} \underline{\RR}$, as an object of  $F^{m}\! \Loc_{\infty}(BG_n)$, is isomorphic to $\CW_{\omega}^{(n)}(V' \rtimes_{\gamma} \RR)$, where $V' \rtimes_{\gamma} \RR$ is the extension of $V'$ by $\gamma$. Hence the desired conclusion follows by taking $V = V' \rtimes_{\gamma} \RR$.  
\end{proof}

Now we are in a position to state and prove the general theorem we have been looking for.

\begin{theoremB}
Given a compact connected Lie group $G$, the $\A_{\infty}$-functor 
$$
\Wcal \colon \InfLoc_{\infty}(\gfrak) \longrightarrow \Loc_{\infty}(BG)
$$
is an $\A_{\infty}$-quasi-equivalence. Moreover, for any principal $G$-bundle $\pi \colon P \to X$ with connection $\theta$  and classifying map $f \colon X \to  BG_n$, there there exists an $\A_{\infty}$-natural isomorphism between the $\A_{\infty}$-functors $\Ical \circ \CW_{\theta}$ and $(\varphi_n \circ f)_{\sbullet}^* \circ \Wcal$ from $ \InfLoc_{\infty}(\gfrak)$ to $\Rep_{\infty}(\pi_{\infty}X_{\sbullet})$. Here $\Ical$ is the integration $\A_\infty$-functor provided by the higher Riemann-Hilbert correspondence, and $\varphi_n$ is the canonical map from $BG_n$ to $BG$.
\end{theoremB}

\begin{proof}
We first prove that the functor $\A_{\infty}$-functor $\Wcal \colon \InfLoc_{\infty}(\gfrak) \to \Loc_{\infty}(BG)$ is $\A_{\infty}$-quasi fully faithful. Let $V$ and $V'$ two objects of $\InfLoc_{\infty}(\gfrak)$. By definition, we have to show that
$$
\Wcal \colon (\W\gfrak \otimes \Hom(V,V'))_{\bas} \longrightarrow \underline{\Hom}^{\sbullet}(\Wcal(V),\Wcal(V'))
$$
is a quasi-isomorphism. Notice that this morphism is compatible with the canonical filtrations on $(\W\gfrak \otimes \Hom(V,V'))_{\bas}$ and $\underline{\Hom}^{\sbullet}(\Wcal(V),\Wcal(V'))$, and induces a morphism between their associated spectral sequences. It therefore suffices to verify that the induced morphism between the second pages of the spectral sequences is an isomorphism. On the one hand, recall from Lemma~\ref{lem:4.12} that the second page of the spectral sequence associated with the canonical filtration on $(\W\gfrak \otimes \Hom(V,V'))_{\bas}$ is
$$
\Ecal_2^{p,q} = \uH^{p}((\W \gfrak)_{\bas}) \otimes \Hcal^{q}(\Hom(V,V')). 
$$
On the other hand, Lemma~\ref{lem:4.10} says that the second page of the spectral sequence associated with the canonical filtration on $\underline{\Hom}^{\sbullet}(\Wcal(V),\Wcal(V'))$ is
$$
\Ecal'^{p,q}_2 = \uH^{p}(\pi_{\infty} BG_{\sbullet}, \Hcal^{q}(\Wcal(V),\Wcal(V'))). 
$$
In addition, since $BG$ is simply connected, the graded vector bundle $\Hcal^{\sbullet}(\Wcal(V),\Wcal(V'))$ is trivial with fiber $\Hcal^{\sbullet}(\Hom(V,V'))$, from which it follows that
$$
\Ecal'^{p,q}_2 \cong \uH^{p}(\pi_{\infty} BG_{\sbullet}) \otimes  \Hcal^{q}(\Hom(V,V')) = \uH^{p}(BG) \otimes  \Hcal^{q}(\Hom(V,V')). 
$$
Consequently, the induced morphism between the second pages of the spectral sequence is
$$
\Wcal \colon \uH^{p}((\W \gfrak)_{\bas}) \otimes \Hcal^{q}(\Hom(V,V')) \longrightarrow \uH^{p}(BG) \otimes  \Hcal^{q}(\Hom(V,V')),
$$
and it is given by the tensor product of the universal Weil homomorphism for $BG$ with the identity on $\Hcal^{q}(\Hom(V,V'))$. Since the former is an isomorphism, the desired assertion follows. 

We next show that the $\A_{\infty}$-functor $\Wcal \colon \InfLoc_{\infty}(\gfrak) \to \Loc_{\infty}(BG)$ is $\A_{\infty}$-quasi essentially surjective. Let us fix an object $(E,F_{\sbullet})$ of $\Loc_{\infty}(BG)$. Then there exists a sufficiently large integer $m > 0$ such that $(E,F_{\sbullet})$ belongs to $F^{m} \! \Loc_{\infty}(BG)$. Since $\V_k(\CC^{n})$ is $2(n-k)$-connected, from the long exact homotopy sequence, we infer that, if $2m < 2(n-k)$, the canonical map $\varphi_n \colon BG_n \to BG$ induces an isomorphism in homotopy group up to degree $2m$. Invoking Lemma~\ref{lem:4.11}, we conclude that the pullback functor $\varphi_{n \sbullet}^* \colon \Loc_{\infty}(BG) \to \Rep_{\infty}(\pi_{\infty} BG_{n\sbullet})$ induces an equivalence between the homotopy categories of   $F^{m} \! \Loc_{\infty}(BG)$ and $F^{m} \! \Rep_{\infty}(\pi_{\infty} BG_{n\sbullet})$. On the other hand, by construction, we have the following commutative diagram of $\A_{\infty}$-functors
$$
\xymatrix{  F^{m} \!\InfLoc_{\infty}(\gfrak) \ar[r]^-{\Wcal} \ar[d]_-{\CW_{\omega}^{(n)}} & F^{m} \! \Loc_{\infty}(BG) \ar[d]^-{\varphi_{n \sbullet}^*} \\
F^{m} \!\Loc_{\infty}(BG_n) \ar[r]^-{\Ical^{(n)}} & F^{m} \! \Rep_{\infty}(\pi_{\infty} BG_{n\sbullet})}
$$
where $\varphi_{n \sbullet}^*$ is the pullback DG functor along the induced simplicial map $\varphi_{n\sbullet} \colon \pi_{\infty} BG_{n\sbullet} \to \pi_{\infty} BG_{\sbullet}$. Thus, since $\Ical^{(n)}$ is an $\A_{\infty}$-quasi-equivalence, it will be enough to show that $\CW_{\omega}^{(n)}$ is a quasi-essentially surjective. But, if $n$ is sufficiently large, this is true by Proposition~\ref{prop:4.15}. 

Finally, we prove the second assertion. To this end, we first notice that by the naturallity of the integration $\A_{\infty}$-functors $\Ical^{(n)}$ and $\Ical$, the following diagram commutes
$$
\xymatrix{\Loc_{\infty}(BG_n) \ar[r]^-{f^*} \ar[d]_-{\Ical^{(n)}} &  \Loc_{\infty}(X) \ar[d]^-{\Ical} \\
 \Rep_{\infty}(\pi_{\infty}BG_{n\sbullet}) \ar[r]^-{f_{\sbullet}^*}& \Rep_{\infty}(\pi_{\infty}X_{\sbullet}),}
$$
where $f_{\sbullet}^*$ is the pullback DG functor along the induced simplicial map $f_{\sbullet} \colon \pi_{\infty}X_{\sbullet} \to \pi_{\infty}BG_{n\sbullet}$. Besides, by the naturality of the connection $\omega$, we get a commutative diagram
$$
\xymatrix{\Loc_{\infty}(BG_n) \ar[rr]^-{f^*} & & \Loc_{\infty}(X) \\
& \InfLoc_{\infty}(\gfrak).  \ar[ul]^-{\CW_{\omega}^{(n)}} \ar[ur]_-{\CW_{f^*\omega}} &}
$$
Combining these last two with the above commutative diagram gives
$$
(\varphi_n \circ f)_{\sbullet}^* \circ \Wcal = f_{\sbullet}^* \circ \varphi_{n\sbullet}^* \circ \Wcal =  f_{\sbullet}^* \circ \Ical^{(n)} \circ \CW_{\omega}^{(n)} = \Ical \circ f^* \circ \CW_{\omega}^{(n)} = \Ical \circ \CW_{f^*\omega}. 
$$
On the other hand, according to Theorem~A, there is an $\A_{\infty}$-natural isomorphism between $\CW_{\theta}$ and $\CW_{f^*\omega}$. Since $\Ical$ is an $\A_{\infty}$-quasi-equivalence, the desired conclusion now follows from Lemma~\ref{lem:2.2}.
\end{proof}


\addcontentsline{toc}{section}{Bibliography}

\end{document}